\documentclass[12,reqno]{amsart}
\usepackage{amsfonts}
\usepackage{subfigure}
\usepackage{graphicx}
\usepackage{amssymb}
\usepackage{amsthm}
\usepackage{amsmath}
\usepackage{dsfont}
\usepackage{a4wide}
\usepackage{pgf,pgfarrows,pgfautomata,pgfheaps,pgfnodes,pgfshade}
\usepackage{url}
\usepackage{bbm}

\usepackage[numbers]{natbib}

\usepackage{hyperref}
\numberwithin{equation}{section}

\newtheorem{defi}{Definition}[section]
\newtheorem{thm}[defi]{Theorem}
\newtheorem{lemm}[defi]{Lemma}
\newtheorem{remark}[defi]{Remark}

\newtheorem{assum}[defi]{Assumption}
\newtheorem{prop}[defi]{Proposition}
\newtheorem{prob}[defi]{Problem}

\begin{document}

	\title[ Robust optimal consumption, investment and reinsurance]{Robust optimal consumption, investment and reinsurance for recursive preferences}
	\author{Elizabeth Dadzie}
	\address{Elizabeth Dadzie: Department of Mathematics, University of Ghana, Accra, LG 25, Ghana; African Institute for Mathematical Sciences, Accra, LG DTD 20046, Ghana}
	\email{edadzie@aims.edu.gh}
	
	\author{Wilfried Kuissi-Kamdem}
	\address{Wilfried Kuissi-Kamdem: Department of Mathematics, University of Rwanda, Kigali, 4285, Rwanda; African Institute for Mathematical Sciences, Accra, LG DTD 20046, Ghana; Department of Mathematical Stochastics, University of Freiburg, Freiburg, 79104, Germany}
	\email{donatien@aims.edu.gh, wilfried.kuissi.kamdem@stochastik.uni-freiburg.de}
	
	\author{Marcel Ndengo}
	\address{Marcel Ndengo: Department of Mathematics, University of Rwanda, Kigali, 4285, Rwanda}
	\email{serandengo@gmail.com}
	\thanks{This work was supported by a grant from the African Institute for Mathematical Sciences, with financial support from the Government of Canada, provided through Global Affairs Canada, and the International Development Research Centre.}
	
	\date{}

	
	\maketitle
	
	\begin{abstract}
		This paper investigates a robust optimal consumption, investment, and reinsurance problem for an insurer with Epstein-Zin recursive preferences operating under model uncertainty. The insurer's surplus follows the diffusion approximation of the Cramér-Lundberg model, and the insurer can purchase proportional reinsurance. Model ambiguity is characterised by a class of equivalent probability measures, and the insurer, being ambiguity-averse, aims to maximise utility under the worst-case scenario. By solving the associated coupled forward-backward stochastic differential equation (FBSDE), we derive closed-form solutions for the optimal strategies and the value function. Our analysis reveals how ambiguity aversion, risk aversion, and the elasticity of intertemporal substitution (EIS) influence the optimal policies. Numerical experiments illustrate the effects of key parameters, showing that optimal consumption decreases with higher risk aversion and EIS, while investment and reinsurance strategies are co-dependent on both financial and insurance market parameters, even without correlation. This study provides a comprehensive framework for insurers to manage capital allocation and risk transfer under deep uncertainty.
	\end{abstract}


\section{Introduction}\label{Introduction}
The optimal management of an insurance's wealth requires balancing between different sources of risk and return: the allocation of funds in the financial market via investment decisions and the transfer of insurance (underwriting) risk through reinsurance. Classical financial economics and actuarial research has studied these problems extensively under expected utility theory; see \cite{schmidli2007stochastic} and reference therein. In this formulation, the insurer (or investor) maximises classical \textit{time-additive} utilities of terminal wealth.

However, from an economics point of view, the main unattractive feature of time-additive preferences is the fact that they fail to separate investors' desire to smooth consumption across states of nature (measured by the coefficient of risk aversion) and investors' willingness to smooth consumption over time (measured by the coefficient of elasticity of intertemporal substitution EIS); see \cite[on pp.227-228]{xing2017consumption} for more details. This limitation has led to a considerable amount of current theoretical and empirical research in finance and economics based on more general dynamic risk preferences.

One of the most popular response in the literature are recursive preferences. Such preferences allow to disentangle the link between risk aversion and EIS; thanks to the postulate that current consumption depend on the value of future consumption. Arguably the most popular among recursive utilities is the Epstein-Zin utility as proposed in \cite{epstein1989substitution}. Since then the Epstein-Zin utility has been widely used in a variety of different contexts covering asset pricing, decision theory, business cycles and	growth, and monetary economics. However, despite the established and rapid growing literature on consumption and investment problems with recursive utilities, to the best of our knowledge no research has ever solved such problems when reinsurance is taken into account.

There is by now ample evidence in the literature that both insurers and investors operate under model uncertainty: the true drift or volatility of asset returns, and the intensity or severity of claims, may not be known with certainty; see \cite{chen2020robust} for a review. In the presence of such ambiguity, a robust decision maker evaluates outcomes under a set of plausible probability measures and maximises utility against the worst-case scenario. Robust control theory (see \cite{hansen2001robust,maenhout2004robust}) integrates this feature by introducing an additional minimisation over alternative measures, penalised by a relative-entropy term. Combining Epstein–Zin utilities with robustness yields robust recursive preferences, which capture both the investors’ intertemporal trade-offs and their concern for model misspecification. For insurers, this provides a realistic framework for studying capital allocation, reinsurance design, and consumption smoothing under deep uncertainty.

In the present paper, we incorporate ambiguity aversion to study the optimal robust consumption (``dividend", ``refund",...), investment and reinsurance problem through maximising, over a finite time-horizon, the Epstein-Zin recursive utility.  A further improvement arises from the fact that we consider an insurer subject to a liability at the end of the investment period. We obtain closed-form solutions for the robust optimal consumption, investment-reinsurance strategy and the corresponding value function by adopting an extension of a well-known technique proposed by \cite{hu2005utility} (for time-additive utility) and \cite{xing2017consumption} (for Epstein-Zin utility). This extension has been introduced in \cite{kk2025optimal} to study a consumption-investment optimisation problem with liability and Epstein-Zin utility under partial information. In order to analyse the effect of ambiguity and the utility's parameters (risk aversion coefficient and EIS coefficient) on the optimal	strategy, we consider three special cases, i.e., uncorrelated claims, without ambiguity, and with ambiguity. Finally, we perform some numerical experiments to illustrate the robust optimal consumption, investment-reinsurance strategy.

The remainder of the present paper is structured as follows. Section \ref{Setting} introduces the financial–insurance market model and the insurer’s wealth dynamics under proportional reinsurance. In Section \ref{Ambiguity-averse framework} we formulate the robust stochastic optimisation problem. In Section \ref{Solution to the robust optimisation problem} we give the main results of this paper. In Section \ref{Numerical illustrations} we perform some numerical analysis. Finally, Section \ref{Conclusion} summarises this paper.

\section{Model and problem formulation}\label{Problem formulation}
\subsection{Probability setting and wealth process of the insurer}\label{Setting}
We consider a filtered probability space $(\Omega,\mathcal{F},(\mathcal{F}_{t})_{0\leq t\leq T},\mathbb{P})$ generated by a $2$-dimensional Brownian motion $B=(W,W^{re})$. The filtration $(\mathcal{F}_{t})_{0\leq t\leq T}$ is assumed to satisfy the usual conditions of completeness and right-continuity, so that we can take c{\`a}dl{\`a}g versions for semi-martingales. We define some known spaces of stochastic processes.
\begin{itemize}
	\item[$(i)$] Let $\mathcal{C}$ be the set of non-negative progressively measurable processes on $[0,T]\times\Omega$.
	\item[$(ii)$] Let $\mathcal{H}_{\mathbb{P}}^{q},~q\ge1$, denotes the space of progressively measurable $\mathbb{R}$-valued processes $(Y_{t})_{0\le t\le T}$ such that $\|Y\|_{\mathcal{H}_{\mathbb{P}}^{q}}=\mathbb{E}[\int_{0}^{T}|Y_{t}|^{q}\mathrm{d}t]^{1/q}<\infty$.
	\item[$(iii)$] Let $\Xi_{\mathbb{P}}^{q},~q\ge1$, denotes the space of predictable $\mathbb{R}^{2}$-valued processes\\ $(Z_{t})_{0\le t\le T}$ such that $\|Z\|_{\Xi_{\mathbb{P}}^{q}}=\mathbb{E}[\exp\big(\frac{q}{2}\int_{0}^{T}\|Z_{t}\|^{2}\mathrm{d}t\big)]^{1/q}<\infty$.
	\item[$(iv)$] Let $\mathbb{H}_{\mathbb{P}}^{q},~q\ge1$, denotes the space of predictable $\mathbb{R}^{2}$-valued processes\\ $(Z_{t})_{0\le t\le T}$ such that $\|Z\|_{\mathbb{H}_{\mathbb{P}}^{q}}=\mathbb{E}[(\int_{0}^{T}\|Z_{t}\|^{2}\mathrm{d}t)^{\frac{q}{2}}]^{1/q}<\infty$.
\end{itemize}

Note that similar spaces can be defined under another probability measure $\mathbb{Q}$, by replacing $\mathbb{P}$ with $\mathbb{Q}$ in the subscripts of the corresponding spaces, and taking expectations with respect to $\mathbb{Q}$.

Now, we can introduce the wealth process, under $\mathbb{P}$, of an insurer. We consider a dynamic financial-insurance environment with two traded assets and the surplus process of the insurer. The traded assets consist of one riskless bond $S^{0}$ and one risky asset $S$ with dynamics
\begin{align}\label{Tradable assets}
	\begin{cases}
		\mathrm{d}S_{t}^{0}&=rS_{t}^{0}\mathrm{d}t,~S_{0}^{0}>0,\\
		\mathrm{d}S_{t}&=S_{t}\left((r+\mu)\mathrm{d}t+\sigma\mathrm{d}W_{t}\right),~S_{0}>0.
	\end{cases}
\end{align}

We assume that, without reinsurance, the surplus process $\widehat{U}$ of the insurer satisfies the diffusion approximation of the classical Cram{\'e}r-Lundberg model (see, e.g., \cite[Sect.~IV.8]{asmussen2020risk} or \cite{ma2023optimal})
\begin{align}
	\mathrm{d}\widehat{U}_{t}&=\kappa\zeta\mathrm{d}t-\sqrt{\kappa\beta}\big(\rho^{S}\mathrm{d}W_{t}+\rho^{re}\mathrm{d}W_{t}^{re}\big)
\end{align}
where $\rho^{S},\rho^{re}\in[-1,1]$ are the correlation coefficients such that $\rho^{re}\ne0$ and $(\rho^{S})^{2}+(\rho^{re})^{2}=1$, $\kappa\zeta$ is the claim rate at $t\in[0,T]$, and $\zeta,\kappa,\beta>0$. The insurance company participates in the reinsurance market and buys proportional reinsurance $\pi_{t}^{re}$ at every time $t\in[0,T]$. As in \cite{bauerle2005benchmark}, the reinsurance strategy $\pi_{t}^{re}$ is allowed to be greater than $1$; expressing the situation in which the insurance company also acts as reinsurer of other insurance companies. At any time $t$, the insurance company retains $100\pi_{t}^{re}\%$ of the total claims while the reinsurer undertakes the rest $100(1-\pi_{t}^{re})\%$. Using expected value principle the insurer and the reinsurer premium rates are determined by $(1+\nu^{in})\kappa\zeta$ and $(1+\nu^{re})\kappa\zeta$, respectively, where $\nu^{in}$ is the safety loading of the insurer and $\nu^{re}$ the safety loading of the reinsurer. We exclude the insurer’s arbitrage opportunity by assuming $\nu^{re}>\nu^{in}$. Hence, the modified dynamics of the insurer's surplus is given by
\begin{align}\label{Surplus process}
	\mathrm{d}U_{t}&=\big((1+\nu^{in})\kappa\zeta-(1-\pi_{t}^{re})(1+\nu^{re})\kappa\zeta\big)-\pi_{t}^{re}\mathrm{d}\widehat{U}_{t}\nonumber\\
	&=\big(\nu^{in}-\nu^{re}+\pi_{t}^{re}\nu^{re}\big)\kappa\zeta\mathrm{d}t+\pi_{t}^{re}\sqrt{\kappa\beta}\big(\rho^{S}\mathrm{d}W_{t}+\rho^{re}\mathrm{d}W_{t}^{re}\big).
\end{align}

In addition to choosing an amount of reinsurance $\pi_{t}^{re}$, $t\in[0,T]$, the insurer also chooses her consumption rate $c_{t}$ (in the form of ``dividend", ``refund",...) and an amount to be invested in the risky assets (investment strategy) $\pi_{t}^{S}$. For such $(c,\pi^{S},\pi^{re})$, the wealth process $\widetilde{X}$ of the company with initial endowment $x\ge0$ at time $0$ evolves according to the stochastic differential equation (SDE)
\begin{align}\label{Wealth_Ambiguity-neutral0}
	\mathrm{d}\widetilde{X}_{t}&=r\widetilde{X}_{t}\mathrm{d}t+\pi_{t}^{S}\mu\mathrm{d}t+\pi_{t}^{S}\mathrm{d}W_{t}-c_{t}\mathrm{d}t+\mathrm{d}U_{t}\nonumber\\
	&=r\widetilde{X}_{t}\mathrm{d}t+\Big(\pi_{t}^{S}\mu+\pi_{t}^{re}\nu^{re}\kappa\zeta\Big)\mathrm{d}t+\big(\nu^{in}-\nu^{re}\big)\kappa\zeta\mathrm{d}t+\pi_{t}^{S}\sigma\mathrm{d}W_{t}\notag\\
	&\phantom{X}+\pi_{t}^{re}\sqrt{\kappa\beta}\big(\rho^{S}\mathrm{d}W_{t}+\rho^{re}\mathrm{d}W_{t}^{re}\big)\notag\\
	&=r\widetilde{X}_{t}\mathrm{d}t+\pi_{t}^{\intercal}\eta\mathrm{d}t+\pi_{t}^{\intercal}\mathrm{d}B_{t}+\big(\nu^{in}-\nu^{re}\big)\kappa\zeta\mathrm{d}t-c_{t}\mathrm{d}t,
\end{align}
where $\Sigma:=\left( \begin{matrix}\sigma & 0 \\ \rho^{S} & \sqrt{\kappa\beta}\rho^{re}\end{matrix} \right)$, $\eta:=\Sigma^{-1}\left( \begin{matrix}\mu \\ \nu^{re}\kappa\zeta\end{matrix} \right)$ and $\pi_{t}^{\intercal}:=\big(\pi_{t}^{S},\pi_{t}^{re}\big)\Sigma,~0\le t\le T$.

As in \cite{ma2023optimal}, 
instead of working with the wealth process $(\widetilde{X}_{t})_{0\le t\le T}$ itself, we consider its self-financing form process given by
\begin{align}\label{Self-financing wealth}
	X_{t}:=\widetilde{X}_{t}+\big(\nu^{in}-\nu^{re}\big)\kappa\zeta\int_{t}^{T}e^{-r(s-t)}\mathrm{d}s~\text{ for ~}t\in[0,T].
\end{align}
Clearly, $X_{T}=\widetilde{X}_{T}$. Hence, Equation~\eqref{Wealth_Ambiguity-neutral0} transforms to
\begin{align}\label{Wealth_Ambiguity-neutral}
	\mathrm{d}X_{t}&=\mathrm{d}\widetilde{X}_{t}+r\big(\nu^{in}-\nu^{re}\big)\kappa\zeta\int_{t}^{T}e^{-r(s-t)}\mathrm{d}s-\big(\nu^{in}-\nu^{re}\big)\kappa\zeta\mathrm{d}t\notag\\
	&=r\widetilde{X}_{t}\mathrm{d}t+\pi_{t}^{\intercal}\eta\mathrm{d}t+\pi_{t}^{\intercal}\mathrm{d}B_{t}-c_{t}\mathrm{d}t+r\big(\nu^{in}-\nu^{re}\big)\kappa\zeta\int_{t}^{T}e^{-r(s-t)}\mathrm{d}s\notag\\
	&=rX_{t}\mathrm{d}t+\pi_{t}^{\intercal}\eta\mathrm{d}t+\pi_{t}^{\intercal}\mathrm{d}B_{t}-c_{t}\mathrm{d}t,
\end{align}
with $X_{0}=x+\big(\nu^{in}-\nu^{re}\big)\kappa\zeta\int_{0}^{T}e^{-rs}\mathrm{d}s$.

\subsection{The consumption, investment and reinsurance problem for an ambiguity-averse insurer}\label{Ambiguity-averse framework}
The framework given in Section \ref{Setting} concerned an insurer who has total confidence in model~\eqref{Wealth_Ambiguity-neutral} under the probability measure $\mathbb{P}$. However, in practice insurers are concerned about model misspecification generated by the deviation from the reference probability measure $\mathbb{P}$. We shall then integrate the probability distribution uncertainty into the consumption-investment-reinsurance optimisation problem of an ambiguity-averse insurer (AAI). To define alternative models, we consider other probability measures---equivalent to the reference measure $\mathbb{P}$---defined, via Radon-Nykodim derivative, by
\begin{align}\label{Alternative probability measure}
	\frac{\mathrm{d}\mathbb{Q}^{\xi}}{\mathrm{d}\mathbb{P}}\Big|_{\mathcal{F}_{T}}:=\exp\Big(-\frac{1}{2}\int_{0}^{T}\|\xi_{s}\|^{2}\mathrm{d}s-\int_{0}^{T}\xi_{s}^{\intercal}\mathrm{d}B_{s}\Big),
\end{align}
where $\xi:=(\xi^{S},\xi^{re})^{\intercal}\in\Xi_{\mathbb{P}}^{2}$ is called the distortion process. According to Girsanov's theorem, we can define on the probability measure $\mathbb{Q}^{\xi}$ the following Brownian motions:
\begin{align}
	&W_{t}^{\mathbb{Q}^{\xi}}:=W_{t}+\int_{0}^{t}\xi_{s}^{S}\mathrm{d}s~\text{ and }~W_{t}^{re,\mathbb{Q}^{\xi}}:=W_{t}^{re}+\int_{0}^{t}\xi_{s}^{re}\mathrm{d}s,
\end{align}
or, equivalently, $B_{t}^{\mathbb{Q}^{\xi}}:=B_{t}+\int_{0}^{t}\xi_{s}\mathrm{d}s$ for $t\in[0,T]$. 

Under $\mathbb{Q}^{\xi}$, the dynamics of the wealth process $X$ in \eqref{Wealth_Ambiguity-neutral} becomes
\begin{align}\label{Wealth_Ambiguity-averse}
	\mathrm{d}X_{t}&=rX_{t}\mathrm{d}t+\pi_{t}^{\intercal}\eta\mathrm{d}t+\pi_{t}^{\intercal}\mathrm{d}B_{t}^{\mathbb{Q}^{\xi}}-c_{t}\mathrm{d}t-\pi_{t}^{\intercal}\xi_{t}\mathrm{d}t.
\end{align}

An AAI's preference over $\mathcal{C}$-valued consumption and $\Xi_{\mathbb{P}}^{2}$-valued distortion  is given by a robust version of the classical continuous-time stochastic differential utility of Epstein-Zin type. To describe this preference, let $\delta>0$ represent the discounting rate, $0<\gamma\neq1$ be the relative risk aversion, and $0<\psi\neq1$ be the elasticity of intertemporal substitution coefficient (EIS). Then, the Epstein–Zin aggregator is defined by
\begin{align}\label{Epstein-Zin generator}
	f(c,v)&:=\delta e^{-\delta t} \frac{c^{1-\frac{1}{\psi}}}{1-\frac{1}{\psi}}((1-\gamma)v)^{1-\frac{1}{\theta}},\text{ with }~\theta:=\frac{1-\gamma}{1-\frac{1}{\psi}},
\end{align}
and the bequest utility function by $h(c):=e^{-\delta\theta T}\frac{c^{1-\gamma}}{1-\gamma}$. Hence, the robust Epstein-Zin utility over the consumption stream $c\in\mathcal{C}$ and the distortion process $\xi\in\Xi_{\mathbb{P}}^{2}$ on a finite time horizon $T$ is a process $(V_{t}^{c,\xi})_{t\in[0,T]}$ which satisfies
\begin{align}\label{Robust Epstein-Zin utility}
	V_{t}^{c,\xi}&=\mathbb{E}_{t}^{\mathbb{Q}^{\xi}}\Big[h(c_{T})+\int_{t}^{T}\Big(f(c_{s},V_{s}^{c,\xi})+\frac{1}{2\Psi_{s}}\|\xi_{s}\|^{2}\Big)\mathrm{d}s\Big]~\text{for }t\in[0,T],
\end{align}
where $(\Psi_{t})_{t\in[0,T]}$ is a non-negative process which captures the AAI's ambiguity aversion. Here, $\mathbb{E}_{t}^{\mathbb{Q}^{\xi}}[\cdot]$ stands for the conditional expectation $\mathbb{E}^{\mathbb{Q}^{\xi}}[\cdot|\mathcal{F}_{t}]$ under $\mathbb{Q}^{\xi}$. Following \cite{maenhout2004robust}, we adopt a homothetic robustness preference by defining $\Psi$ via
\begin{align}\label{Maenhout ambiguity function}
	\Psi_{t}:=\frac{\Phi}{(1-\gamma)V_{t}^{c,\xi}}~\text{ for }t\in[0,T],
\end{align}
with $\Phi\ge0$ denoting the ambiguity aversion parameter. Hence, the robust recursive utility process $V^{c,\xi}$ in \eqref{Robust Epstein-Zin utility} becomes 
\begin{align}\label{Epstein-Zin utility_Maenhout's style}
	V_{t}^{c,\xi}&=\mathbb{E}_{t}^{\mathbb{Q}^{\xi}}\Big[h(c_{T})+\int_{t}^{T}\Big(f(c_{s},V_{s}^{c,\xi})+\frac{1}{2\Phi}\|\xi_{s}\|^{2}(1-\gamma)V_{s}^{c,\xi}\Big)\mathrm{d}s\Big],~0\le t\le T.
\end{align}

For the analysis in our paper, we study the case
\begin{align}\label{Parameters specifications}
	\gamma>1~\text{ and }~\psi>1.
\end{align}
Our interest in the parameter specification in \eqref{Parameters specifications} originates mainly from its empirical evidence on consumption and portfolio decisions; see \cite[on p.228]{xing2017consumption}.

Without the distortion term in the generator of \eqref{Epstein-Zin utility_Maenhout's style} (no ambiguity), existence and uniqueness results are well-established (see \cite[Prop.~2.2]{xing2017consumption}). To guarantee the existence of a \textit{suitable} unique solution to \eqref{Epstein-Zin utility_Maenhout's style}, for non-zero distortion term, we consider the following set of admissible consumption and distortion streams.
\begin{align}\label{Admissible consumption and distortion processes}
	\mathcal{A}_{a}:=\Big\{(c,\xi)\in\mathcal{C}\times\Xi_{\mathbb{Q}^{\xi}}^{2}~\big|&~\mathbb{E}^{\mathbb{Q}^{\xi}}\Big[\int_{0}^{T}e^{-\delta s}c_{s}^{1-\frac{1}{\psi}}\mathrm{d}s\Big]<\infty~\text{ and }~ \mathbb{E}^{\mathbb{Q}^{\xi}}\big[e^{\int_{0}^{T}\frac{1}{2\Phi} \|\xi_{s}\|^{2}\mathrm{d}s}c_{T}^{1-\gamma}\big]<\infty\Big\}.
\end{align}
\begin{prop}\label{Non-empty control set}
	Suppose $\gamma,\psi>1$ and $(c,\xi)\in\mathcal{A}_{a}$. Then \eqref{Epstein-Zin utility_Maenhout's style} admits a unique solution $V^{c,\xi}$, with $V^{c,\xi}$ continuous, strictly negative and of class (D). Moreover, there exists a square integrable process $Z^{c,\xi}$ such that for $t\in[0,T]$,
	\begin{align}\label{Epstein-Zin utility_Maenhout's style_Integral form}
		V_{t}^{c,\xi}&=h(c_{T})+\int_{t}^{T}\Big(f(c_{s},V_{s}^{c,\xi})+\frac{1}{2\Phi}\|\xi_{s}\|^{2}(1-\gamma)V_{s}^{c,\xi}\Big)\mathrm{d}s-\int_{t}^{T}Z_{t}^{c,\xi}\mathrm{d}B_{s}^{\mathbb{Q}^{\xi}}.
	\end{align}
\end{prop}
\begin{proof}
	See Appendix~\ref{Non-empty control set_Proof}.
\end{proof}

In this section, we are interested in the optimal consumption, investment and reinsurance problem of an AAI with a constant liability $G\in\mathbb{R}$ at terminal time $T$ and robust recursive preference of Epstein-Zin type. Note that $G$ is not necessarily positive. Hence, we want to find the best strategy $(\widehat{c},\widehat{\pi},\widehat{\xi})$ solution to the optimization problem
\begin{align}\label{Problem_Ambiguity-averse}
	\sup_{c,\pi}~\inf_{\xi}~\mathbb{E}^{\mathbb{Q}^{\xi}}\Big[h(X_{T}^{c,\pi,\xi}-G)+\int_{0}^{T}\Big(f(c_{s},V_{s}^{c,\xi})+\frac{1}{2\Phi}\|\xi\|^{2}(1-\gamma)V_{s}^{c,\xi}\Big)\mathrm{d}s\Big],
\end{align}
where $X^{c,\pi,\xi}$ denotes the solution to the SDE~\eqref{Wealth_Ambiguity-averse} associated to the consumption $c$, the investment-reinsurance strategy $\pi$ and the distortion process $\xi$, with $\pi^{\intercal}:=\big((\pi^{S})^{\intercal},\pi^{re}\big)\Sigma$ (see the definition of $\Sigma$ just below \eqref{Wealth_Ambiguity-neutral0}).

To define the set of admissible consumption, investment, reinsurance and distortion strategies, we introduce the BSDE
\begin{align}\label{Auxiliary BSDE_AAI}
	\mathrm{d}Y_{t}&=-\big(\mathcal{H}(t,X_{t}^{c,\pi,\xi},Y_{t},Z_{t})+Z_{t}^{\intercal}\xi_{t}\big)\mathrm{d}t+Z_{t}^{\intercal}\mathrm{d}B_{t}^{\mathbb{Q}^{\xi}},\quad Y_{T}=-e^{-rT}G,
\end{align}
where the function $\mathcal{H}$ is to be defined. Hence, we define the set of admissible consumption, investment, reinsurance and distortion strategies as follows. 
\begin{defi}\label{Admissible strategies}
	A triple $(c,\pi,\xi)$ of consumption, investment-reinsurance and distortion strategies is \textit{admissible} if
	\begin{itemize}
		\item[1.] $(c,\xi)\in\mathcal{A}_{a}$ with $c_{T}=X_{T}^{c,\pi,\xi}+e^{rT}Y_{T}$;
		\item[2.] $X_{t}^{c,\pi,\xi}+e^{rt}Y_{t}>0$ for all $t\in[0,T]$;
		\item[3.] $(X_{\cdot}^{c,\pi,\xi}+e^{r\cdot}Y_{\cdot})^{1-\gamma}$ is of class (D) under $\mathbb{P}$.
	\end{itemize}
\end{defi}
We denote by $\mathcal{A}^{AAI}$ the set of admissible consumption, investment-reinsurance and distortion strategies. Therefore, we are interested in the following problem:
\begin{prob}\label{Mainpb_AAI}
	Find $(\widehat{c},\widehat{\pi},\widehat{\xi})\in\mathcal{A}^{AAI}$ such that 
	\begin{align}\label{General problem_AAI}
		&\mathcal{V}^{AAI}:=V_{0}^{\widehat{c},\widehat{\pi},\widehat{\xi}}:=\underset{(c,\pi,\xi)\in\mathcal{A}^{AAI}}{\sup~\inf}\mathbb{E}^{\mathbb{Q}^{\xi}}\Big[h(X_{T}^{c,\pi,\xi}-G)+\int_{0}^{T}\Big(f(c_{s},V_{s}^{c,\xi})+\frac{1}{2\Phi}\|\xi\|^{2}(1-\gamma)V_{s}^{c,\xi}\Big)\mathrm{d}s\Big].
	\end{align}
\end{prob}

\section{Solution to the AAI’s stochastic optimisation problem}\label{Solution to the robust optimisation problem}
We speculate that the optimal utility process takes the form
\begin{align}\label{EZ utility_Ansatz}
	\widehat{V}_{t}={e^{-\delta\theta t}}\frac{(X_{t}+e^{rt}Y_{t})^{1-\gamma}}{1-\gamma},\quad0\le t\le T,
\end{align}
where $(Y,Z)$ is the solution to the BSDE \eqref{Auxiliary BSDE_AAI}. We define the process
\begin{align}\label{Optimal martingale}
	M_{\cdot}^{c,\pi,\xi}&:={e^{-\delta\theta t}}\frac{(X_{\cdot}+e^{r\cdot}Y_{\cdot})^{1-\gamma}}{1-\gamma}\notag\\
	&\phantom{xx}+\int_{0}^{\cdot}\Big(f\big(c_{s},e^{-\delta\theta s}\frac{(X_{s}+e^{rs}Y_{s})^{1-\gamma}}{1-\gamma}\big)+\frac{1}{2\Phi}\|\xi\|^{2}(X_{s}+e^{rs}Y_{s})^{1-\gamma}\Big)\mathrm{d}s.
\end{align}
From the martingale optimality principle, the function $\mathcal{H}$ in \eqref{Auxiliary BSDE_AAI} must be chosen according to the following rules:	
\begin{itemize}
	\item[$(1)$] For any $(c,\pi)$, the process $M^{c,\pi,\xi}$ is a local submartingale for all $\xi$ such that $(c,\pi,\xi)\in\mathcal{A}^{AAI}$.
	\item[$(2)$] For any $\xi$, the process $M^{c,\pi,\xi}$ is a local supermartingale for all $(c,\pi)$ such that $(c,\pi,\xi)\in\mathcal{A}^{AAI}$.
	\item[$(3)$] There exists a $(\widehat{c},\widehat{\pi},\widehat{\xi})\in\mathcal{A}^{AAI}$ such that $M^{\widehat{c},\widehat{\pi},\widehat{\xi}}$ is a local martingale.
\end{itemize}

Recall $f$ defined in \eqref{Epstein-Zin generator}. Set $\mathcal{H}_{t}:=\mathcal{H}(t,X_{t},Y_{t},Z_{t})$ for all $t\in[0,T]$. To find $\mathcal{H}$, we apply It{\^o}'s formula to $M^{c,\pi,\xi}$ in \eqref{Optimal martingale} and obtain
\begin{align}\label{OptEZ_Differential_AAI}
	\mathrm{d}M_{t}^{c,\pi,\xi}&=-\delta\theta{e^{-\delta\theta t}}\frac{(X_{t}+e^{rt}Y^{0})^{1-\gamma}}{1-\gamma}\mathrm{d}t+re^{rt}Y_{t}e^{-\delta\theta t}(X_{t}+e^{rt}Y_{t})^{-\gamma}\mathrm{d}t\nonumber\\
	&+{e^{-\delta\theta t}}(X_{t}+e^{rt}Y_{t})^{-\gamma}\mathrm{d}X_{t}+e^{rt}{e^{-\delta\theta t}}(X_{t}+e^{rt}Y_{t})^{-\gamma}\mathrm{d}Y_{t}\nonumber\\
	&-\frac{\gamma}{2}{e^{-\delta\theta t}}(X_{t}+e^{rt}Y_{t})^{-\gamma-1}(\mathrm{d}X_{t})^{2}-\frac{\gamma}{2}e^{2rt}{e^{-\delta\theta t}}(X_{t}+e^{rt}Y_{t})^{-\gamma-1}(\mathrm{d}Y_{t})^{2}\nonumber\\
	&-\gamma e^{rt}{e^{-\delta\theta t}}(X_{t}+e^{rt}Y_{t})^{-\gamma-1}\mathrm{d}X_{t}\mathrm{d}Y_{t}+f\big(c_{t},{e^{-\delta\theta t}}\frac{(X_{t}+e^{rt}Y_{t})^{1-\gamma}}{1-\gamma}\big)\mathrm{d}t\nonumber\\
	&=-\delta\theta e^{-\delta\theta t}\frac{(X_{t}+e^{rt}Y_{t})^{1-\gamma}}{1-\gamma}\mathrm{d}t+e^{-\delta\theta t}(X_{t}+e^{rt}Y_{t})^{-\gamma}rX_{t}\mathrm{d}t\nonumber\\
	&+{e^{-\delta\theta t}}(X_{t}+e^{rt}Y_{t})^{-\gamma}\pi_{t}^{\intercal}\eta \mathrm{d}t+e^{-\delta\theta t}(X_{t}+e^{rt}Y_{t})^{-\gamma} \pi_{t}^{\intercal}\mathrm{d}B_{t}^{\mathbb{Q}^{\xi}}\nonumber\\
	&-e^{-\delta\theta t}(X_{t}+e^{rt}Y_{t})^{-\gamma}c_{t}\mathrm{d}t -e^{-\delta\theta t}(X_{t}+e^{rt}Y_{t})^{-\gamma}\pi_{t}^{\intercal}\xi_{t} \mathrm{d}t\nonumber\\
	&-e^{rt}e^{-\delta\theta t}(X_{t}+e^{rt}Y_{t})^{-\gamma}\mathcal{H}_{t} \mathrm{d}t-e^{rt}e^{-\delta\theta t}(X_{t}+e^{rt}Y_{t})^{-\gamma} Z_{t}^{\intercal}\xi_{t}\mathrm{d}t\nonumber\\
	&+e^{rt}e^{-\delta\theta t}(X_{t}+e^{rt}Y_{t})^{-\gamma}Z_{t}^{\intercal} \mathrm{d}B_{t}^{\mathbb{Q}^{\xi}}-\frac{\gamma}{2}e^{-\delta\theta t}(X_{t}+e^{rt}Y_{t})^{-\gamma-1}\pi_{t}^{\intercal}\pi_{t}\mathrm{d}t\nonumber\\
	&-\frac{\gamma}{2}e^{2rt}e^{-\delta\theta t}(X_{t}+e^{rt}Y_{t})^{-\gamma-1}\|Z_{t}\|^{2}\mathrm{d}t-\gamma e^{rt}e^{-\delta\theta t}(X_{t}+e^{rt}Y_{t})^{-\gamma-1}\pi_{t}^{\intercal} Z_{t}\mathrm{d}t\nonumber\\
	&+\delta\frac{c_{t}^{1-\frac{1}{\psi}}}{1-\frac{1}{\psi}}e^{-\delta\theta t} (X_{t}+e^{rt}Y_{t})^{-\gamma+\frac{1}{\psi}}\mathrm{d}t+\frac{1}{2\Phi}\|\xi\|^{2}e^{-\delta\theta t}(X_{t}+e^{rt}Y_{t})^{1-\gamma}\mathrm{d}t\notag\\
	&+re^{rt}Y_{t}e^{-\delta\theta t}(X_{t}+e^{rt}Y_{t})^{-\gamma}\mathrm{d}t \nonumber\\
	&=\Big[\Big(-c_{t}+\delta\frac{c_{t}^{1-\frac{1}{\psi}}}{1-\frac{1}{\psi}}(X_{t}+e^{rt}Y_{t})^{\frac{1}{\psi}}\Big){e^{-\delta\theta t}} (X_{t}+e^{rt}Y_{t}) ^{-\gamma}\nonumber\\
	&-\frac{\gamma+\Phi}{2}e^{-\delta\theta t}(X_{t}+e^{rt}Y_{t})^{-\gamma-1} \Big(\pi_{t}^{\intercal}\pi_{t}+2\pi_{t}^{\intercal}\Big(e^{rt}Z_{t}{-\frac{1}{\gamma+\Phi}}(X_{t}+e^{rt}Y_{t})\eta\Big)\Big)\nonumber\\
	&-\frac{\gamma+\Phi}{2}e^{2rt}e^{-\delta\theta t}(X_{t}+e^{rt}Y_{t})^{-\gamma-1} \|Z_{t}\|^{2}-e^{rt}e^{-\delta\theta t}(X_{t}+e^{rt}Y_{t})^{-\gamma} \mathcal{H}_{t}\nonumber\\
	&+re^{rt}Y_{t}e^{-\delta\theta t} (X_{t}+e^{rt}Y_{t})^{-\gamma}+e^{-\delta\theta t}(X_{t}+e^{rt}Y_{t})^{-\gamma}rX_{t}\nonumber\\
	&-\delta\theta e^{-\delta\theta t}\frac{(X_{t}+e^{rt}Y_{t})^{1-\gamma}} {1-\gamma}\Big]\mathrm{d}t+e^{-\delta\theta t}(X_{t}+e^{rt}Y_{t})^{-\gamma} \big(\pi_{t}^{\intercal}+e^{rt}Z_{t}^{\intercal}\big)\mathrm{d}B_{t} ^{\mathbb{Q}^{\xi}}\notag\\
	&+\frac{1}{2\Phi}e^{-\delta\theta t}(X_{t}+e^{rt}Y_{t})^{1-\gamma} \Big(\|\xi_{t}\|^{2}-2\Phi(X_{t}+e^{rt}Y_{t})^{-1}\big(\pi_{t}^{\intercal}+e^{rt}Z_{t}^{\intercal}\big)\xi_{t}\Big)\mathrm{d}t\notag\\
    &=e^{-\delta\theta t}(X_{t}+e^{rt}Y_{t})^{-\gamma}\Bigg[-c_{t}+\delta \frac{c_{t}^{1-\frac{1}{\psi}}}{1-\frac{1}{\psi}}(X_{t}+e^{rt}Y_{t})^{\frac{1}{\psi}}-e^{rt}Z_{t}^{\intercal}\eta\nonumber\\
	&+\frac{1}{2\Phi}(X_{t}+e^{rt}Y_{t})\big\|\xi_{t}-\Phi(X_{t}+e^{rt}Y_{t})^{-1}\big(\pi_{t}+e^{rt}Z_{t}\big)\big\|^{2}\notag\\
	&-\frac{\gamma+\Phi}{2}(X_{t}+e^{rt}Y_{t})^{-1}\Big\|\pi_{t}+\Big(e^{rt}Z_{t}{-\frac{1}{\gamma+\Phi}}(X_{t}+e^{rt}Y_{t})\eta\Big)\Big\|^{2}\nonumber\\
	&+\frac{1}{2}\frac{1}{\gamma+\Phi}(X_{t}+e^{rt}Y_{t})\|\eta\|^{2}+r(X_{t}+e^{rt}Y_{t})-\frac{\delta\theta}{1-\gamma}(X_{t}+e^{rt}Y_{t})-e^{rt}\mathcal{H}_{t}\Bigg]\mathrm{d}t\nonumber\\
	&+e^{-\delta\theta t}(X_{t}+e^{rt}Y_{t})^{-\gamma}\big(\pi_{t}^{\intercal} +e^{rt}Z_{t}^{\intercal}\big)\mathrm{d}B_{t}^{\mathbb{Q}^{\xi}}.
\end{align}
Applying the rules $1,2$ and $3$ above, we expect that $(1)$ for any $(c,\pi)$, the drift in \eqref{OptEZ_Differential_AAI} is non-negative for all $\xi$, $(2)$ for any $\xi$, the drift in \eqref{OptEZ_Differential_AAI} is non-positive for all $(c,\pi)$, and $(3)$ the drift in \eqref{OptEZ_Differential_AAI} is zero for the optimal triple $(\widehat{c},\widehat{\pi},\widehat{\xi})$. Hence, the generator $\mathcal{H}$ for \eqref{Auxiliary BSDE_AAI} can be obtained by formally taking the infimum over $\xi$ and a supremum over $c$ and $\pi$ in the drift in \eqref{OptEZ_Differential_AAI} and setting it to be zero. That is,
\begin{align}\label{YDrift_AAI}
	\mathcal{H}(t,X_{t},Y_{t},Z_{t})&=e^{-rt}\max_{c}\Big\{-c_{t}+\delta\frac{c_{t}^{1-\frac{1}{\psi}}}{1-\frac{1}{\psi}}(X_{t}+e^{rt}Y_{t})^{\frac{1}{\psi}}\Big\}\notag\\
	&+\min_{\xi}\Big\{\frac{1}{2\Phi}e^{-rt}(X_{t}+e^{rt}Y_{t})\big\|\xi_{t}-\Phi(X_{t}+e^{rt}Y_{t})^{-1}\big(\pi_{t}+e^{rt}Z_{t}\big)\big\|^{2}\}\notag\\
	&+\max_{\pi}\Big\{-\frac{\gamma+\Phi}{2}e^{-rt}(X_{t}+e^{rt}Y_{t})^{-1}\Big\|\pi_{t}+\Big(e^{rt}Z_{t}{-\frac{1}{\gamma+\Phi}}(X_{t}+e^{rt}Y_{t})\eta\Big)\Big\|^{2}\Big\}\nonumber\\
	&+\frac{1}{2}\frac{1}{\gamma+\Phi}e^{-rt}(X_{t}+e^{rt}Y_{t})\|\eta\|^{2}+re^{-rt}(X_{t}+e^{rt}Y_{t})-\frac{\delta\theta}{1-\gamma}e^{-rt}(X_{t}+e^{rt}Y_{t}).
\end{align}
Therefore, we deduce from the three optimisation problems in \eqref{YDrift_AAI} that the candidate optimal consumption $\widehat{c}$, the candidate optimal investment-reinsurance $\widehat{\pi}$ and the candidate optimal distortion process $\widehat{\xi}$ are given by
\begin{align}
	\widehat{c}_{t}&=\delta^{\psi}(X_{t}+e^{rt}Y_{t})\label{Candidate consumption_AAI}\\
	\widehat{\pi}_{t}&=-e^{rt}Z_{t}+\frac{1}{\gamma+\Phi}(X_{t}+e^{rt}Y_{t})\eta\label{Candidate investment_AAI}\\
	\widehat{\xi}_{t}&=\Phi(X_{t}+e^{rt}Y_{t})^{-1}\big(\widehat{\pi}_{t}+e^{rt}Z_{t}\big)=\frac{\Phi}{\gamma+\Phi}\eta,\label{Candidate distortion process}
\end{align}
Hence, substituting \eqref{Candidate consumption_AAI}, \eqref{Candidate investment_AAI} and \eqref{Candidate distortion process} into \eqref{Wealth_Ambiguity-averse} and \eqref{YDrift_AAI}, the function $\mathcal{H}$ and the wealth process $X=:\widehat{X}$ are given by
\begin{align}
	&\mathcal{H}(t,\widehat{X}_{t},Y_{t},Z_{t})=e^{-rt}\Big(\frac{\delta^{\psi}}{\psi-1}+r+\frac{1}{2}\frac{1}{\gamma+\Phi}\|\eta\|^{2}{-\frac{\delta\theta}{1-\gamma}}\Big)(\widehat{X}_{t}+e^{rt}Y_{t})-Z_{t}^{\intercal}\eta\label{Auxiliary BSDE generator_AAI}\\
	&\text{and }\mathrm{d}\widehat{X}_{t} =\Big(r\widehat{X}_{t}+\big(-\delta^{\psi}{+\frac{1}{\gamma+\Phi}}\|\eta\|^{2}\big)(\widehat{X}_{t}+e^{rt}Y_{t})-{\frac{\gamma}{\gamma+\Phi}}e^{rt}Z_{t}^{\intercal}\eta\Big)\mathrm{d}t\nonumber\\
	&\phantom{XXXX}+\Big({\frac{1}{\gamma+\Phi}}(\widehat{X}_{t}+e^{rt}Y_{t})\eta^{\intercal}-e^{rt}Z_{t}^{\intercal}\Big)\mathrm{d}B_{t}^{\mathbb{Q}^{\widehat{\xi}}}.\label{Wealth optimal_AAI}
\end{align}
Thus, a candidate solution to Problem \ref{Mainpb_AAI} is given by \eqref{Candidate consumption_AAI}, \eqref{Candidate investment_AAI} and \eqref{Candidate distortion process}, provided that the FBSDE
\begin{align}\label{Optimal FBSDE}
	\begin{cases}
		\mathrm{d}\widehat{X}_{t}&=\Big(r\widehat{X}_{t}+\big(-\delta^{\psi}{+\frac{1}{\gamma+\Phi}}\|\eta\|^{2}\big)(\widehat{X}_{t}+e^{rt}Y_{t})-{\frac{\gamma}{\gamma+\Phi}}e^{rt}Z_{t}^{\intercal}\eta\Big)\mathrm{d}t\\
		&\phantom{X}+\Big({\frac{1}{\gamma+\Phi}}(\widehat{X}_{t}+e^{rt}Y_{t})\eta^{\intercal}-e^{rt}Z_{t}^{\intercal}\Big)\mathrm{d}B_{t}^{\mathbb{Q}^{\widehat{\xi}}}.\\
		\mathrm{d}Y_{t}&=-\Big(e^{-rt}\Big(\frac{\delta^{\psi}}{\psi-1}+r+\frac{1}{2}\frac{1}{\gamma+\Phi}\|\eta\|^{2}{-\frac{\delta\theta}{1-\gamma}}\Big)(\widehat{X}_{t}+e^{rt}Y_{t})-{\frac{\gamma}{\gamma+\Phi}}Z_{t}^{\intercal}\eta\Big)\mathrm{d}t\\
		&\phantom{X}+Z_{t}^{\intercal}\mathrm{d}B_{t}^{\mathbb{Q}^{\widehat{\xi}}}\\
		\widehat{X}_{0}&=x+\big(\nu^{in}-\nu^{re}\big)\kappa\zeta\int_{0}^{T}e^{-rs}\mathrm{d}s,~Y_{T}=-e^{-rT}G
	\end{cases}
\end{align}
is well-defined in an appropriate function space. In the sequel, to simplify the notations, we introduce the process $H^{t}=\{H_{s}^{t},t\le s\le T\}$ defined by
\begin{align}\label{Deflator}
	H_{s}^{t}&:=\mathcal{E}(\int-\eta^{\intercal}\mathrm{d}B)_{s}\big/\mathcal{E}(\int-\eta^{\intercal}\mathrm{d}B)_{t},~t\le s\le T,
\end{align}
with $H:=H^{t}$ for $t=0$, and the process $\varphi=\{\varphi_{t},0\leq t\le T\}$ given by 
\begin{align}\label{Auxiliary process_AAI}
	\varphi_{t}&:=\exp\Big(\Big(-\frac{\delta^{\psi}\psi}{\psi-1}+{\frac{\gamma+3\Phi-1}{2(\gamma+\Phi)^{2}}}\|\eta\|^{2}{+\frac{\delta\theta}{1-\gamma}}\Big)t+\frac{1}{\gamma+\Phi}\eta^{\intercal}B_{t}\Big),~t\le s\le T.
\end{align}
We can now confirm the well-definedness of the FBSDE \eqref{Optimal FBSDE}.
\begin{prop}\label{Existence result on the FBSDE}
	Let $r_{m}$ and $\widetilde{x}$ denote the constants defined by\\ $r_{m}:=-\frac{\delta^{\psi}\psi}{\psi-1}-r-\frac{\gamma-\Phi}{2(\gamma+\Phi)^{2}}\|\eta\|^{2}+\frac{\delta\theta}{1-\gamma}$ and $\widetilde{x}:=\frac{r_{m}\big(x+\frac{\kappa\zeta}{r}\big(\nu^{in}-\nu^{re}\big)\big(1-e^{-rT}\big)-e^{-rT}G\big)}{r_{m}+\big(r_{m}+\delta^{\psi}-\frac{\Phi}{(\gamma+\Phi)^{2}}\|\eta\|^{2}\big)\big(e^{r_{m}T}-1\big)}$. Assume that $\widetilde{x}$ is finite. Then a solution $(\widehat{X},Y,Z)\in\mathcal{H}_{\mathbb{P}}^{q}\times\mathcal{H}_{\mathbb{P}}^{q}\times\mathbb{H}_{\mathbb{P}}^{q},~q\ge1$, to the FBSDE \eqref{Optimal FBSDE} is given by
	{\small
		\begin{align}\label{Solution to the FBSDE_AAI}
			\begin{cases}
				\widehat{X}_{t}&=\widetilde{x}\varphi_{t}-e^{rt}Y_{t}\\
				Y_{t}&=-e^{-rT}G\\
				&\phantom{X}+\widetilde{x}\Big(-r_{m}-\delta^{\psi}+\frac{\Phi}{(\gamma+\Phi)^{2}}\|\eta\|^{2}\Big)\frac{e^{r_{m}T}-e^{r_{m}t}}{r_{m}}\exp\Big(\frac{-1+2(\gamma+\Phi)}{2(\gamma+\Phi)^{2}}\|\eta\|^{2}t+\frac{1}{\gamma+\Phi}\eta^{\intercal}B_{t}\Big)\\
				Z_{t}&=\frac{1}{\gamma+\Phi}\big(Y_{t}+e^{-rT}G\big)\eta.
			\end{cases}
	\end{align}}
	
	Moreover, the solution $(\widehat{X},Y,Z)$ given by \eqref{Solution to the FBSDE_AAI} is the unique local solution to the FBSDE \eqref{Optimal FBSDE}.
\end{prop}
\begin{proof}
	See Appendix \ref{Solution to the FBSDE_Proof}.
\end{proof}
\begin{remark}
	Note that for $\Phi=0$, the constant $\widetilde{x}$ ($=:\widetilde{x}^{0}$) in Proposition \ref{Existence result on the FBSDE} is finite. Indeed, 
	\begin{align*}
		\widetilde{x}^{0}&=\frac{r_{m}\big(x+\frac{\kappa\zeta}{r}\big(\nu^{in}-\nu^{re}\big)\big(1-e^{-rT}\big)-e^{-rT}G\big)}{r_{m}+\big(r_{m}+\delta^{\psi}\big)\big(e^{r_{m}T}-1\big)}\notag\\
		&=\frac{x+\frac{\kappa\zeta}{r}\big(\nu^{in}-\nu^{re}\big)\big(1-e^{-rT}\big)-e^{-rT}G}{e^{r_{m}T}+\frac{\delta^{\psi}}{r_{m}}\big(e^{r_{m}T}-1\big)}\notag\\
		&=\frac{x+\frac{\kappa\zeta}{r}\big(\nu^{in}-\nu^{re}\big)\big(1-e^{-rT}\big)-e^{-rT}G}{e^{r_{m}T}+\int_{0}^{T}\delta^{\psi}e^{r_{m}s}\mathrm{d}s}.
	\end{align*}
	Because $e^{r_{m}T}+\int_{0}^{T}\delta^{\psi}e^{r_{m}s}\mathrm{d}s>0$, we have $\widetilde{x}^{0}:=\widetilde{x}$ finite.
\end{remark}

To ensure the optimality of the candidate strategies given by \eqref{Candidate consumption_AAI}, \eqref{Candidate investment_AAI} and \eqref{Candidate distortion process} we consider the following conditions.
\begin{assum}\label{Well-definedness of the reinsurance strategy}
	$\widetilde{x}>0$, $\frac{-\rho^{S}\mu+\sigma\nu^{re}\kappa\zeta}{\sigma}>0$ and $-r_{m}-\delta^{\psi}+\frac{\Phi}{(\gamma+\Phi)^{2}}\|\eta\|^{2}<0$.
\end{assum}
\begin{remark}\label{Particular case where xtilde is positive}
	Note that when the liability is non-negative (that is, $G\ge0$), then Assumption \ref{Well-definedness of the reinsurance strategy} yields that the process $Y$, given in \eqref{Solution to the FBSDE_AAI}, is negative. As a by-product, we obtain that the optimal wealth process $\widehat{X}$, given in \eqref{Solution to the FBSDE_AAI}, is positive. Indeed, for $G\ge0$, suppose Assumption \ref{Well-definedness of the reinsurance strategy} holds. Then the right side of the second equality in \eqref{Solution to the FBSDE_AAI} is negative; because $\frac{e^{r_{m}T}-e^{r_{m}t}}{r_{m}}>0$ for all $r_{m}\in\mathbb{R}$. Hence, $Y_{t}<0,~t\in[0,T]$. Using the first equality in \eqref{Solution to the FBSDE_AAI} and the fact that $\widetilde{x}>0$ and $\varphi_{t}>0,~t\in[0,T]$, we deduce that $\widehat{X}_{t}>0$ for all $t\in[0,T]$.
\end{remark}
We are now ready to give the main result of the present paper.
\begin{thm}\label{Mainresult_AAI}
	Suppose Assumption \ref{Well-definedness of the reinsurance strategy} holds. Let $\widetilde{x}$ be defined as in Proposition~\ref{Existence result on the FBSDE}. Let $\varGamma$ be the process defined by $\varGamma_{t}:=-\frac{1}{\gamma+\Phi}e^{-rt}\big(Y_{t}+e^{-rT}G\big)+\frac{1}{\gamma+\Phi}\widetilde{x}\varphi_{t}$ for $t\in[0,T]$. Then the robust optimal consumption $\widehat{c}$, distortion process $\widehat{\xi}$, investment $\widehat{\pi}^{S}$ and reinsurance $\widehat{\pi}^{re}$ strategies given by
	\begin{align}\label{Optimal strategy_AAI}
		\begin{cases}
			\widehat{c}_{t}&=\delta^{\psi}\widetilde{x}\varphi_{t},~~\widehat{\xi}_{t}=\frac{\Phi}{\gamma+\Phi}\frac{1}{\sigma\rho^{re}\sqrt{\kappa\beta}}\left( \begin{matrix}\mu\rho^{re}\sqrt{\kappa\beta} \\ -\rho^{S}\mu+\sigma\nu^{re}\kappa\zeta\end{matrix} \right)\\
			\widehat{\pi}_{t}^{S}&=\frac{1}{\kappa\beta\sigma^{2}(\rho^{re})^{2}}\Big(\mu\kappa\beta(\rho^{re})^{2}-\rho^{S}\Big(-\rho^{S}\mu+\sigma\nu^{re}\kappa\zeta\Big)\Big)\varGamma_{t}\\
			\widehat{\pi}_{t}^{re}&=\frac{1}{\kappa\beta\sigma(\rho^{re})^{2}}\Big(-\rho^{S}\mu+\sigma\nu^{re}\kappa\zeta\Big)\varGamma_{t}
		\end{cases}
	\end{align}
	solve the control problem \eqref{General problem_AAI}, and their associated value function is given by
	\begin{align}\label{Value function_Example}
		\mathcal{V}^{AAI}&=\frac{1}{1-\gamma}\left(\frac{r_{m}\big(x+\frac{\kappa\zeta}{r}\big(\nu^{in}-\nu^{re}\big)\big(1-e^{-rT}\big)-e^{-rT}G\big)}{r_{m}+\big(r_{m}+\delta^{\psi}-\frac{\Phi}{(\gamma+\Phi)^{2}}\|\eta\|^{2}\big)\big(e^{r_{m}T}-1\big)}\right)^{1-\gamma}.
	\end{align}
\end{thm}
\begin{remark}
	Note, from \eqref{Optimal strategy_AAI} and the definition of the vector $\eta$ just below \eqref{Wealth_Ambiguity-neutral0}, that the robust optimal reinsurance strategy (RORS) depends on the parameters of the financial market. Similarly, the parameters of the insurance market impact both the robust optimal consumption strategy (ROCS) and the robust optimal investment strategy (ROIS). This co-dependence happens even if we assume no correlation (meaning, $\rho^{S}=0$) between the insurance market and the financial market.
\end{remark}
We state four preliminaries results, Lemma \ref{Admissiblity of the candidate controls}, \ref{Comparaison theorem}, \ref{Upper bound_Value function} and \ref{Auxiliary martingale}, from which the proof of Theorem \ref{Mainresult_AAI} will follow (see Appendix \ref{Mainresult_AAI_Proof}). Lemma \ref{Admissiblity of the candidate controls} confirms that the candidate controls are admissible and their optimality is shown via Lemma \ref{Upper bound_Value function}.
\begin{lemm}\label{Admissiblity of the candidate controls}
	Recall $(\widehat{X},Y)$ given by \eqref{Solution to the FBSDE_AAI}. Let Assumption \ref{Well-definedness of the reinsurance strategy} holds. Then
	\begin{itemize}
		\item[$(i)$] $\widehat{X}_{t}+e^{rt}Y_{t}>0$ for all $t\in[0,T]$.
		\item[$(ii)$] $(\widehat{c},\widehat{\xi})\in\mathcal{A}_{a}$ and $(\widehat{X}+e^{rt}Y)^{1-\gamma}$ is of class (D) under $\mathbb{P}$.
	\end{itemize}
\end{lemm}
\begin{proof}
	See Appendix \ref{Mainresult_AAI_Proof}.
\end{proof}
\begin{lemm}\label{Comparaison theorem}
	Let $(Y,Z)$ (respectively, $(\widetilde{Y},\widetilde{Z})$) be a super-solution (respectively, sub-solution) to \eqref{Epstein-Zin utility_Maenhout's style_Integral form}. That is,
	\begin{align*}
		&Y+\int_{0}^{\cdot}\Big(f(c_{s},Y_{s})+\frac{1}{2\Phi}\|\xi_{s}\|^{2}(1-\gamma)Y_{s}\Big)\mathrm{d}s\text{ is a local super-martingale and}\\
		&\widetilde{Y}+\int_{0}^{\cdot}\Big(f(c_{s},\widetilde{Y}_{s})+\frac{1}{2\Phi}\|\xi_{s}\|^{2}(1-\gamma)\widetilde{Y}_{s}\Big)\mathrm{d}s\text{ is a local sub-martingale}
	\end{align*}
	with $Y_{T}\ge e^{-\delta\theta T}\frac{c_{T}^{1-\gamma}}{1-\gamma}\ge\tilde{Y}_{T}$, where $Z$ and $\tilde{Z}$ are determined by the Doob–Meyer decomposition	and martingale representation. Assume that both $Y$ and $\tilde{Y}$ are of class (D). Then $Y_{t}\ge\tilde{Y}_{t}$ for $t\in[0,T]$. Moreover, if $Y_{T}>\tilde{Y}_{T}$, then $Y_{t}>\tilde{Y}_{t}$ for $t\in[0,T]$.
\end{lemm}
\begin{proof}
	See Appendix \ref{Mainresult_AAI_Proof}.
\end{proof}
\begin{lemm}\label{Upper bound_Value function}
	Let Assumption \ref{Well-definedness of the reinsurance strategy} holds. Then for any triple $(c,\pi,\xi)$ of admissible strategy we have
	\begin{align}\label{Upper bound_Value function1}
		\frac{(x+Y_{0})^{1-\gamma}}{1-\gamma}\ge V_{0}^{c,\xi},
	\end{align}
	with $c$ financed by $\pi$ via \eqref{Wealth_Ambiguity-averse}, $V^{c,\xi}$ defined in Proposition \ref{Non-empty control set} and $Y$ given in \eqref{Solution to the FBSDE_AAI}.
\end{lemm}
\begin{proof}
	See Appendix \ref{Mainresult_AAI_Proof}.
\end{proof}
\begin{lemm}\label{Auxiliary martingale}
	Let $\widetilde{M}$ be the process defined by
	\begin{align}\label{Auxiliary martingale process}
		\widetilde{M}_{t}:=\exp\Big(\int_{0}^{t}\big(\delta^{\psi}\theta+\frac{\Phi(1-\gamma)}{2(\gamma+\Phi)^{2}}\|\eta\|^{2}e^{\delta\theta s}\big)\mathrm{d}s\Big)e^{-\delta\theta t} \frac{(\widehat{X}_{t}+e^{rt}Y_{t})^{1-\gamma}}{1-\gamma},~0\le t\le T,
	\end{align}
	with $(\widehat{X},Y)$ as in Proposition~\ref{Existence result on the FBSDE}. Then the process $\widetilde{M}$ is a martingale under $\mathbb{Q}^{\widehat{\xi}}$.
\end{lemm}
\begin{proof}
	See Appendix \ref{Mainresult_AAI_Proof}.
\end{proof}

\section{Numerical simulations}\label{Numerical illustrations}
The goal of this section is to numerically illustrate the effects of model parameters on the optimal consumption, investment and reinsurance strategies, and the corresponding value function. We consider three cases: the no-correlation between insurance market and financial market case, the case of an ambiguity-neutral insurer (ANI) and the general case obtained in Theorem \ref{Mainresult_AAI}. For the numerical experiments, except otherwise stated, the basic parameters are chosen as those in Table \ref{table:Basic parameter values}.
\begin{table}[h!]
	\centering
	\begin{tabular}{c}
		\hline
		$\begin{array}{c}
			r \\
			0.05
		\end{array}$~$\begin{array}{c}
			\mu \\
			0.04
		\end{array}$~$\begin{array}{c}
			\sigma \\
			0.25
		\end{array}$~$\begin{array}{c}
			\delta \\
			0.08
		\end{array}$~$\begin{array}{c}
			\gamma \\
			5
		\end{array}$~$\begin{array}{c}
			\psi \\
			1.5
		\end{array}$~$\begin{array}{c}
			T \\
			10
		\end{array}$~$\begin{array}{c}
			x \\
			500
		\end{array}$\\
		\hline
		$\begin{array}{c}
			\kappa \\
			1.5
		\end{array}$~$\begin{array}{c}
			\zeta \\
			1
		\end{array}$~$\begin{array}{c}
			\beta \\
			1
		\end{array}$~$\begin{array}{c}
			\rho^{S} \\
			-\frac{1}{2}
		\end{array}$~$\begin{array}{c}
			\rho^{re} \\
			\frac{\sqrt{3}}{2}
		\end{array}$~$\begin{array}{c}
			\nu^{in} \\
			0.2
		\end{array}$~$\begin{array}{c}
			\nu^{re} \\
			0.5
		\end{array}$~$\begin{array}{c}
			\Phi \\
			2
		\end{array}$~$\begin{array}{c}
			G \\
			500
		\end{array}$ \\
		\hline
	\end{tabular}
	\caption{Values of model parameters.}
	\label{table:Basic parameter values}
\end{table}
\begin{figure}[h!]
	\centering
	\subfigure[]{
		\includegraphics[width=0.45\textwidth]{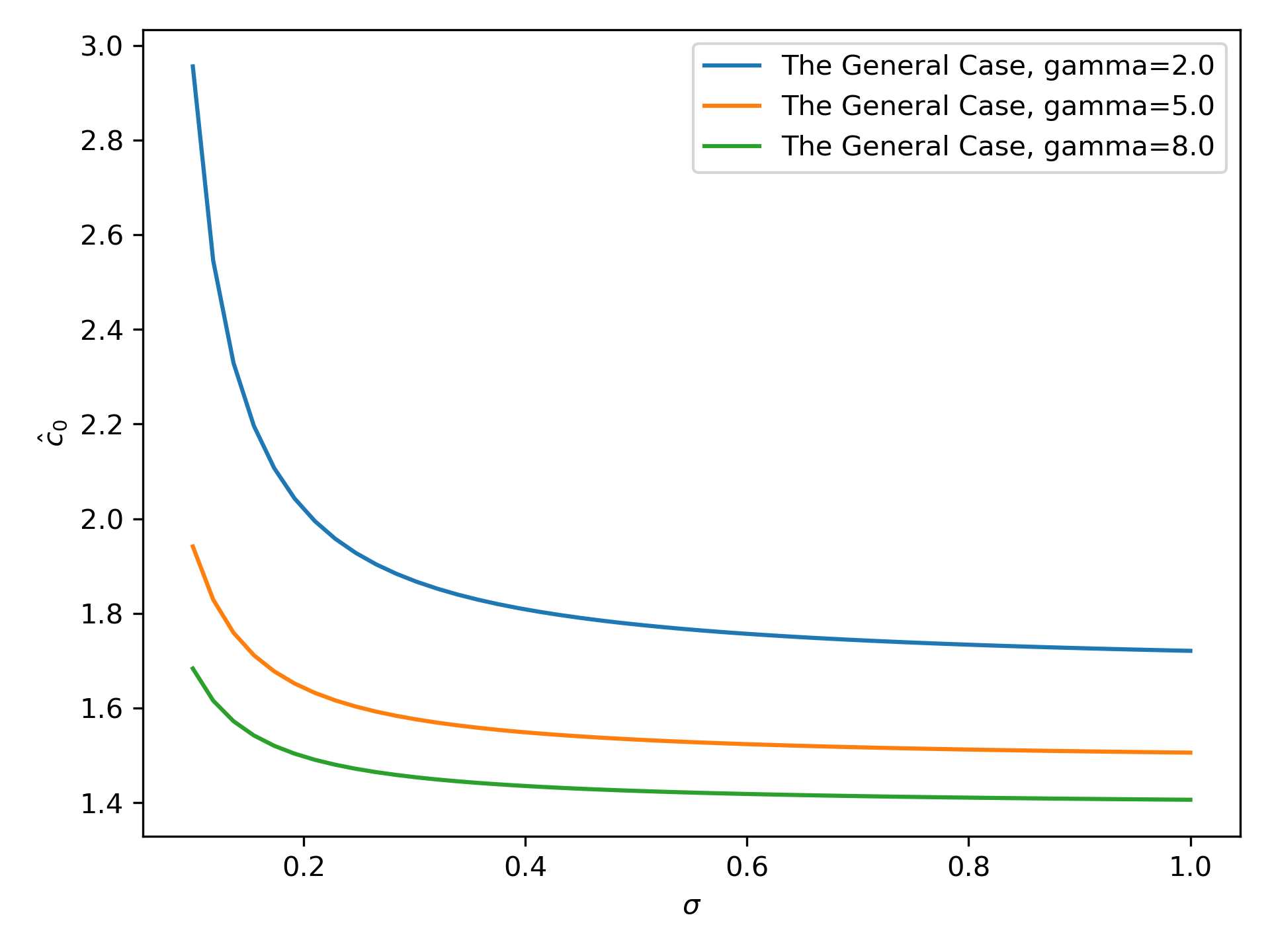}
		\label{fig:AAI consumption vs varying volatility with varying risk aversion}
	}
	\subfigure[]{
		\includegraphics[width=0.45\textwidth]{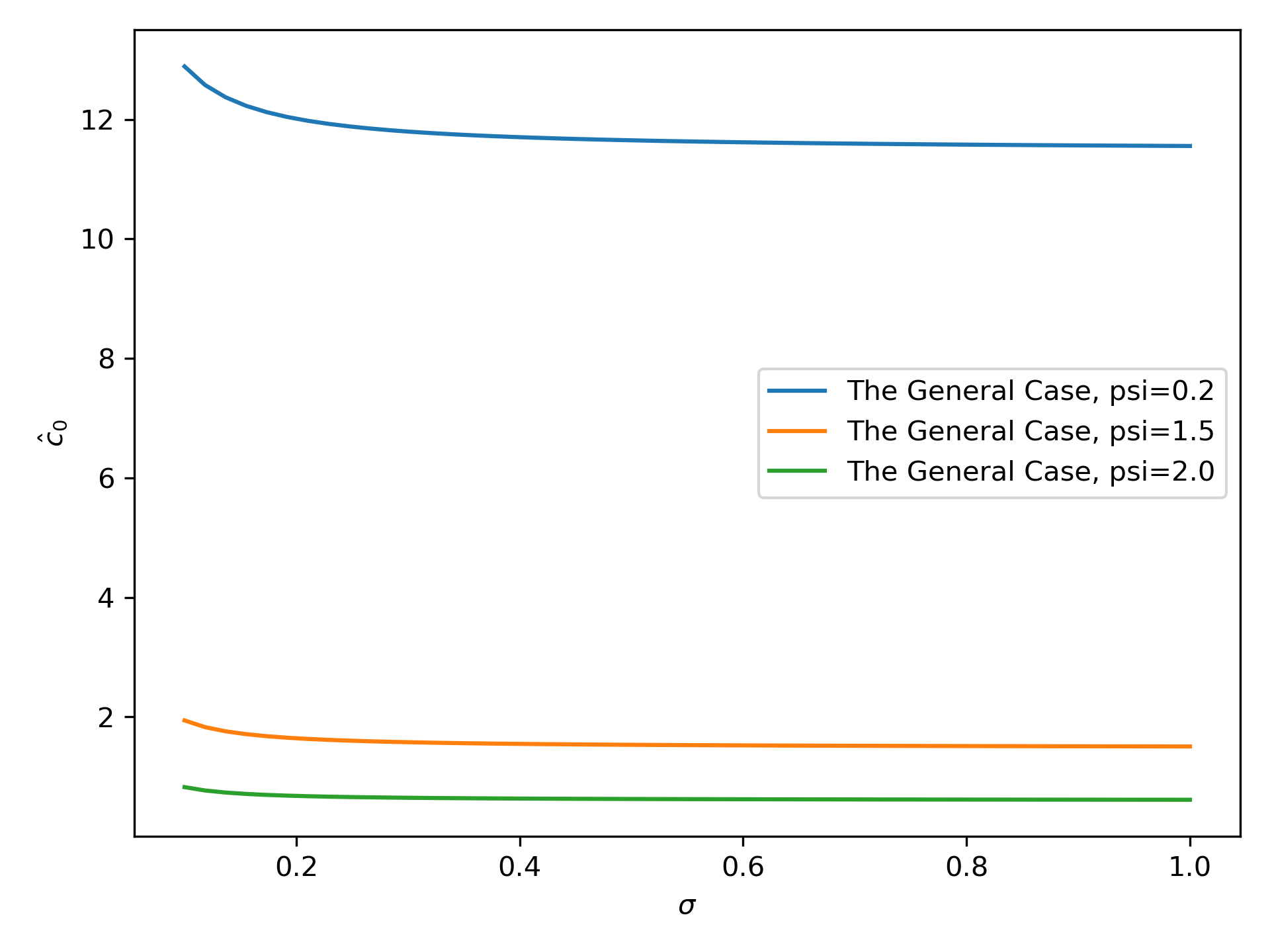}
		\label{fig:AAI consumption vs volatility with varying EIS}
	}
	\caption{The time-$0$ optimal consumption for an ambiguity-averse insurer with correlation between insurance market and financial market (general case). The left panel uses $\psi=1.5$, and the right panel takes $\gamma=5$.}
	\label{fig:AAI consumption strategy}
\end{figure}

In Figure \ref{fig:AAI consumption strategy} we display the time-$0$ robust optimal consumption strategy (ROCS) with respect to the volatility of the stock for different values of the risk aversion (see Figure \ref{fig:AAI consumption vs varying volatility with varying risk aversion}) and the EIS (see Figure \ref{fig:AAI consumption vs volatility with varying EIS}). We observe that the risk aversion coefficient and the EIS coefficient both negatively impact the consumption. In addition, Figure \ref{fig:AAI consumption vs varying volatility with varying risk aversion} shows that the ROCS is highly sensitive to small values of the volatility of the stock ($\sigma<0.2$) and barely influenced by its high values ($\sigma\ge0.4$).
\begin{figure}[h!]
	\centering
	\subfigure[]{
		\includegraphics[width=0.45\textwidth]{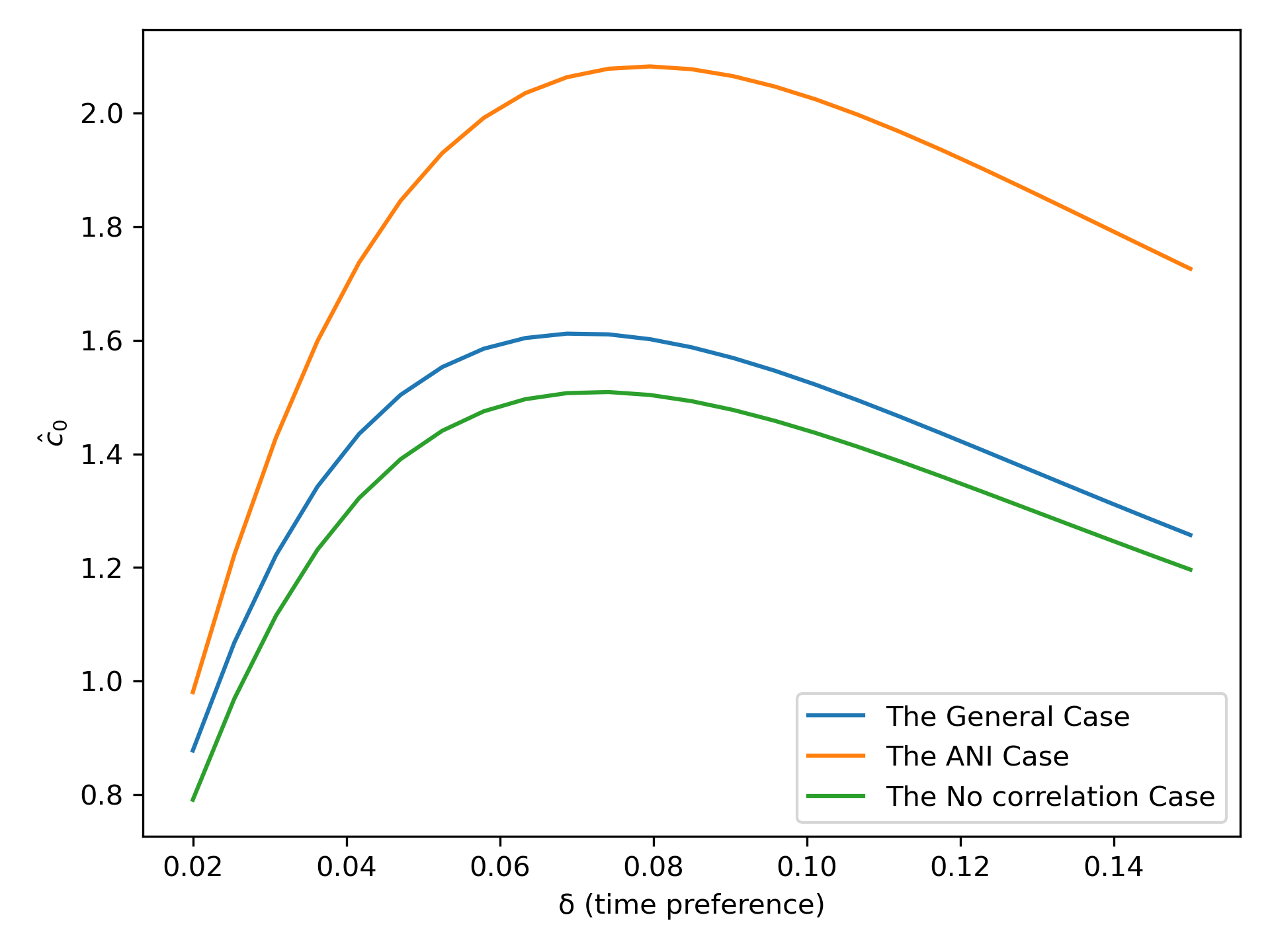}
		\label{fig:All consumptions vs discount factor}
	}
	\subfigure[]{
		\includegraphics[width=0.45\textwidth]{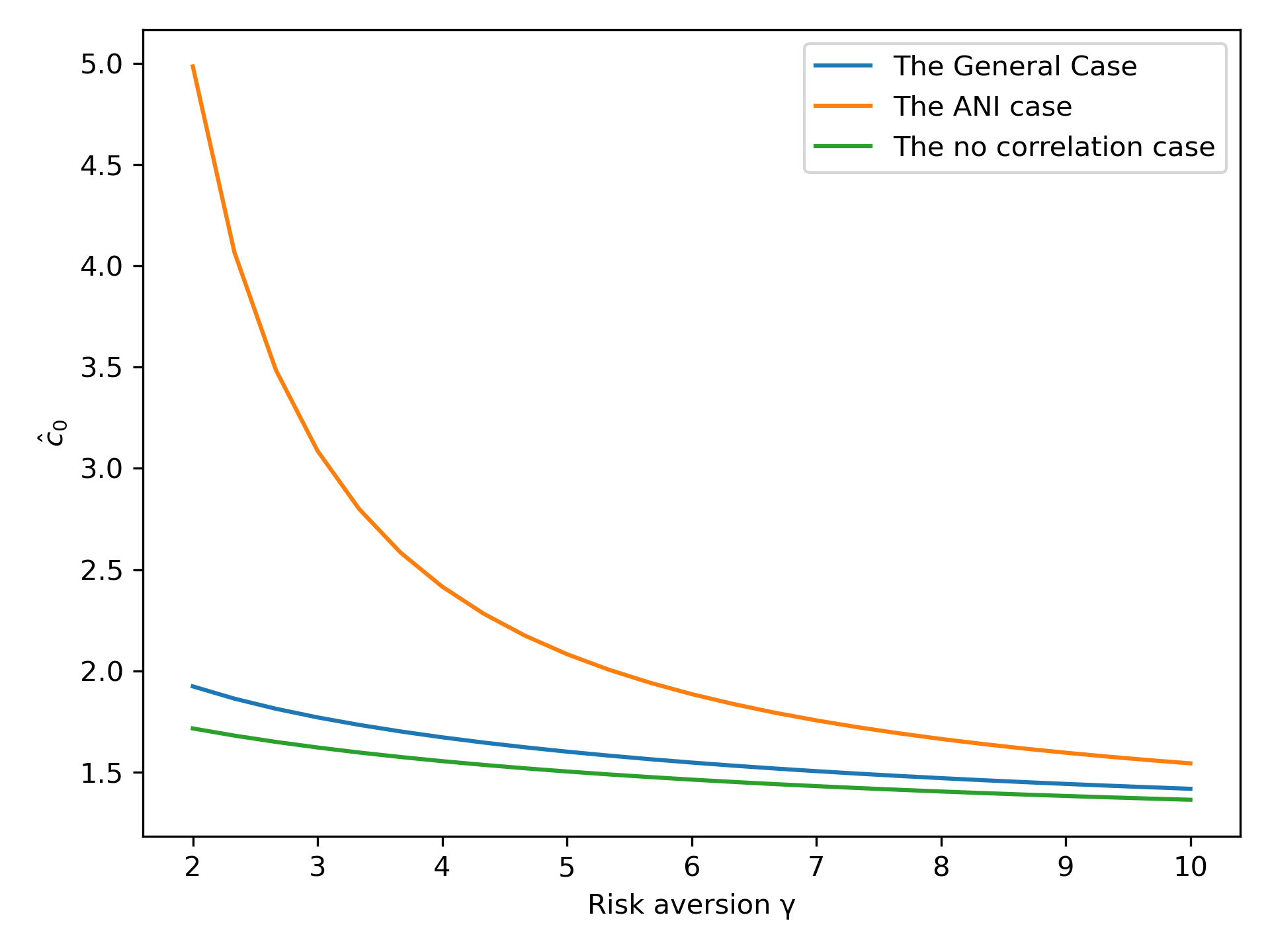}
		\label{fig:All consumption vs risk aversion}
	}
	\caption{The time-$0$ optimal consumption for an ambiguity-neutral insurer (ANI case) and an ambiguity-averse insurer when considering correlation (General case) or no-correlation (No-correlation case) between financial and insurance risks.}
	\label{fig:All consumption strategies}
\end{figure}

Next, to better understand the effect of ambiguity aversion and correlation between financial and insurance risks on the optimal consumption strategy, we display in Figure \ref{fig:All consumption strategies} three cases, i.e., the ambiguity-neutral case ($\Phi=0$), the no-correlation between financial and insurance risks case ($\rho^{S}=0$), and the general case which is determined through the first equation in \eqref{Optimal strategy_AAI}.

In Figure \ref{fig:All consumption strategies} we display the effects of model uncertainty and correlation between financial and insurance risks on the optimal consumption strategy with respect to the discount rate/time preference (see Figure \ref{fig:All consumptions vs discount factor}) and the risk aversion (see Figure \ref{fig:All consumption vs risk aversion}). Figure \ref{fig:All consumptions vs discount factor} shows an inverted U-shape in all three cases which indicates a non-monotonic relationship between the patience level--measured by the discount rate $\delta$-- of insurers and their consumption. The ambiguity-neutral case dominates with highest consumption throughout all cases followed by the general case. In all cases, there is peak consumption with varying values at $\delta\approx 0.07$. Figure \ref{fig:All consumption vs risk aversion} shows a declining consumption with increasing risk aversion for all three curves. Our numerical results show a similar effect of the EIS coefficient on the consumption.
\begin{figure}[h!]
	\centering
	\subfigure[]{
		\includegraphics[width=0.45\textwidth]{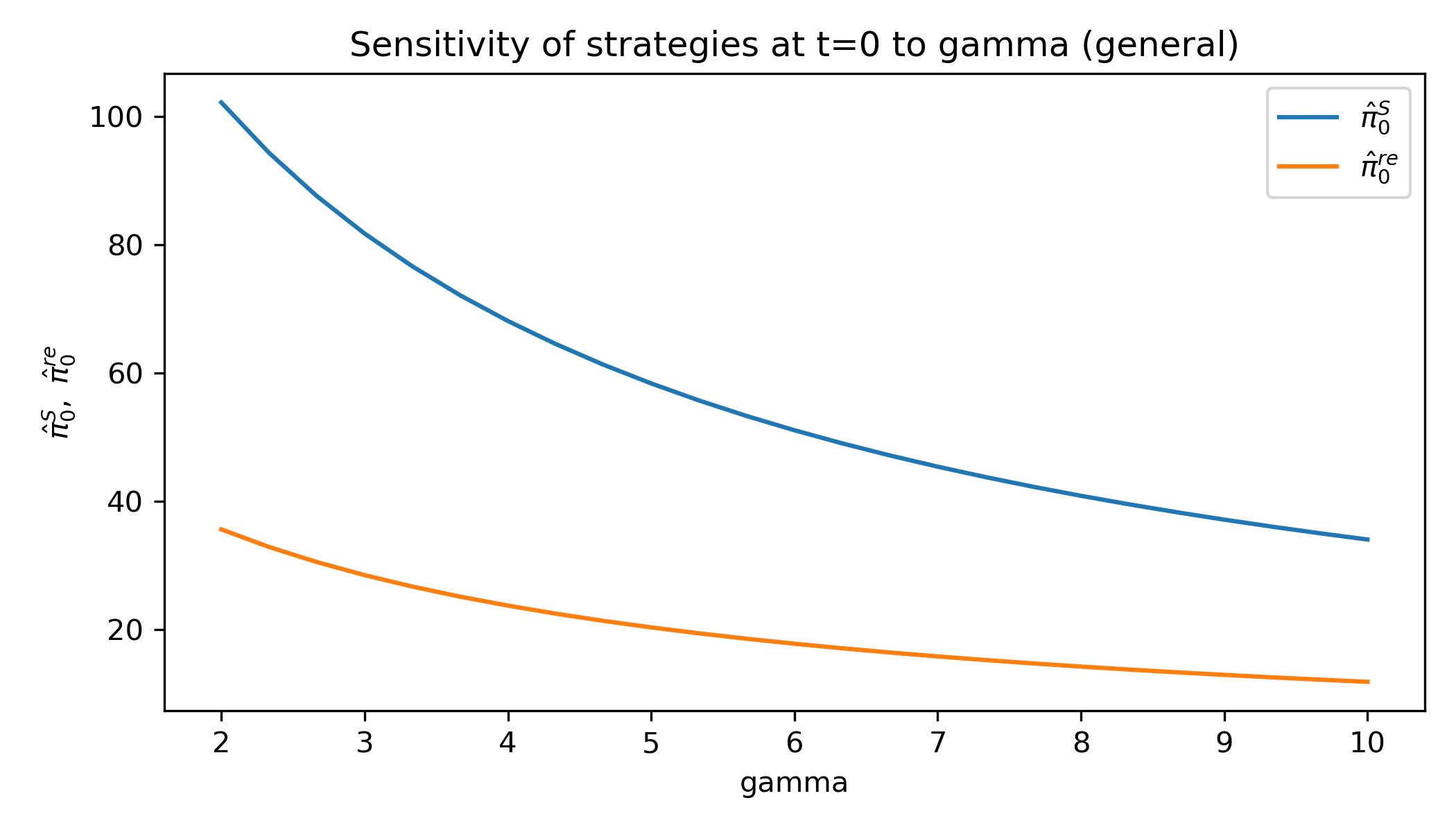}
		\label{fig:Investment-reinsurance vs gamma under uncertainty}
	}
	\subfigure[]{
		\includegraphics[width=0.45\textwidth]{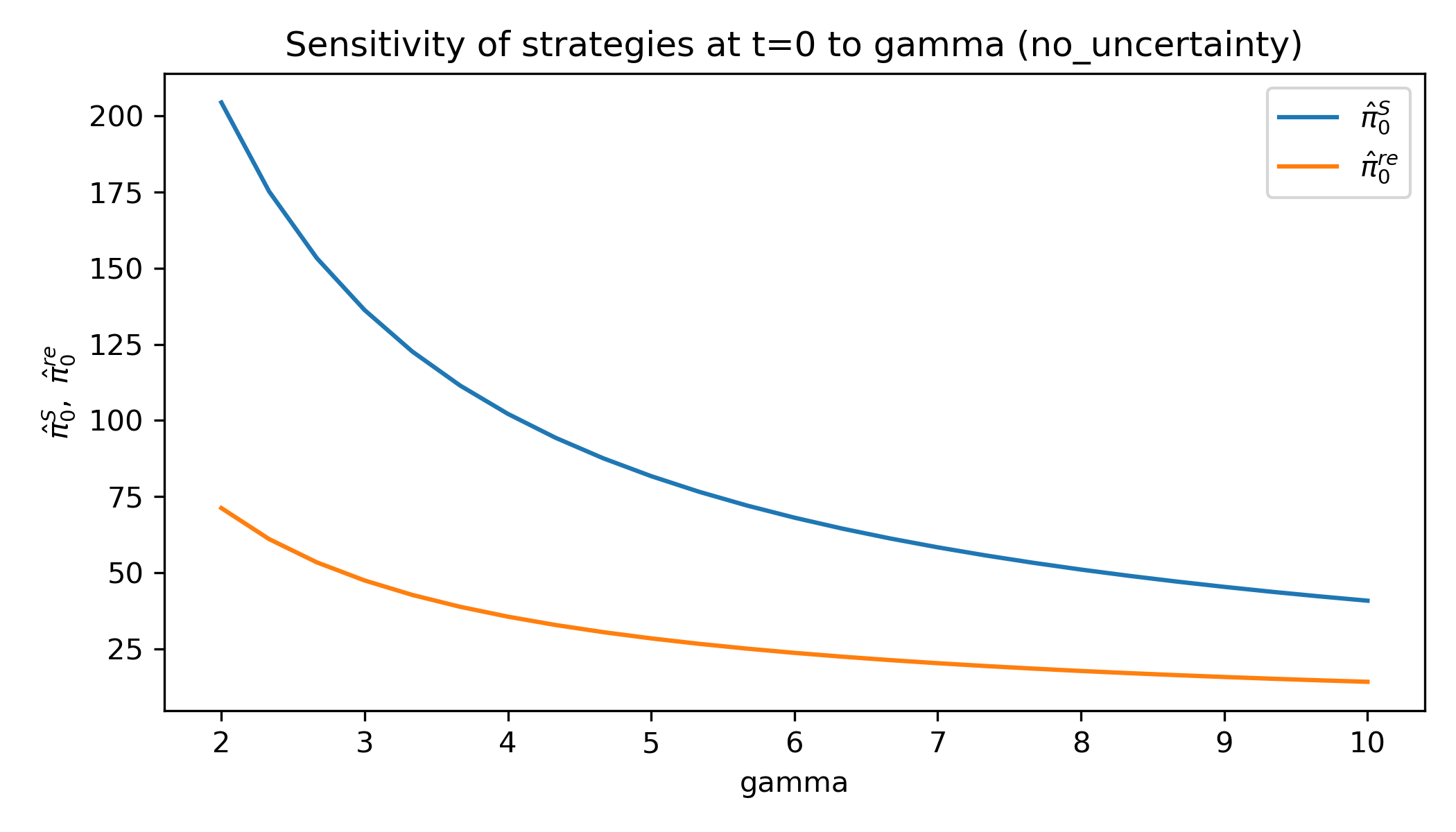}
		\label{fig:Investment-reinsurance vs gamma without correlation}
	}\vfill
	\subfigure[]{
		\includegraphics[width=0.45\textwidth]{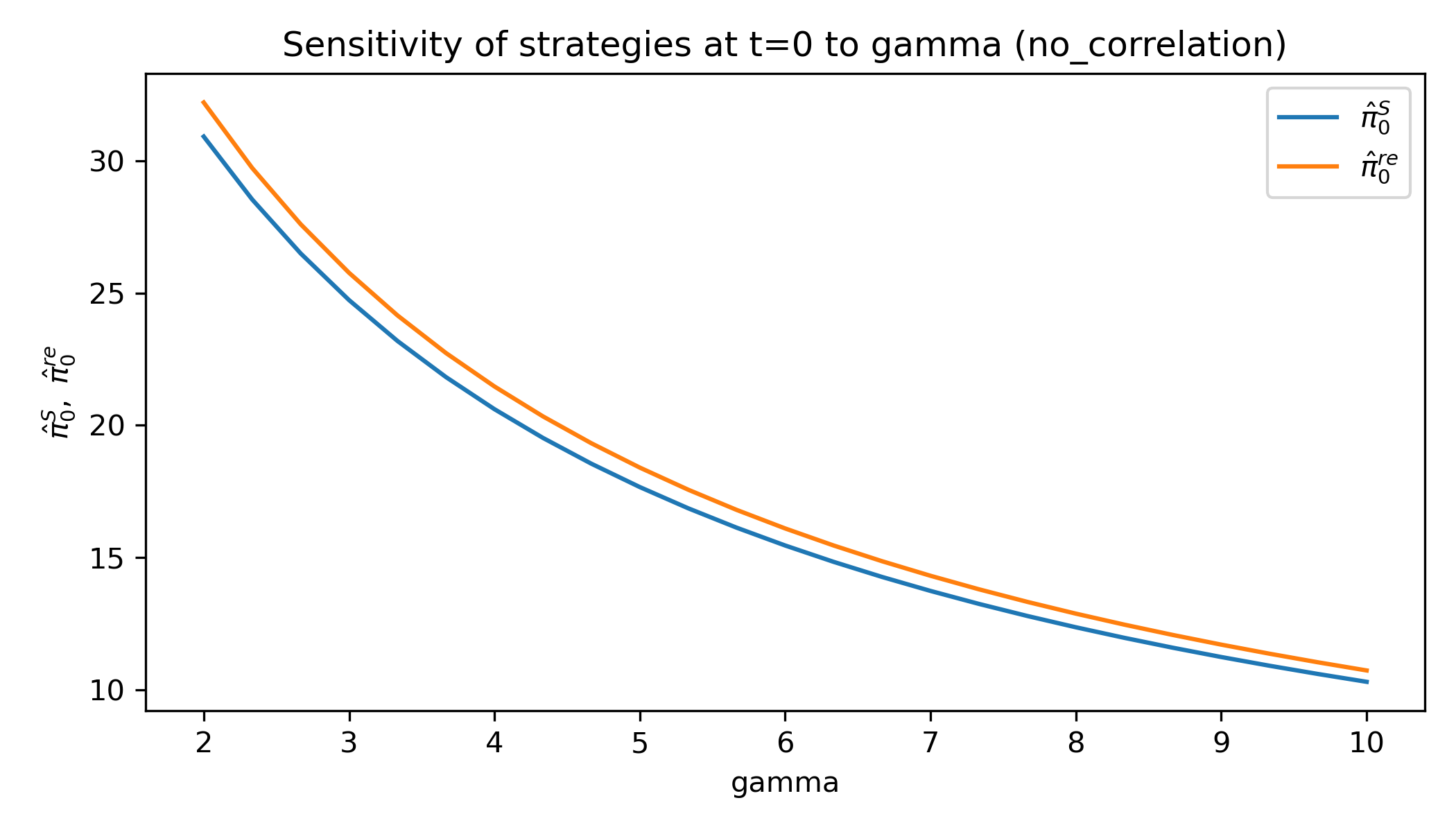}
		\label{fig:Investment-reinsurance vs gamma without uncertainty}
	}
	\caption{The time-$0$ optimal investment and reinsurance  with respect to the risk aversion for an ambiguity-neutral insurer (ANI case) and an ambiguity-averse insurer when considering correlation (General case) or no-correlation (No-correlation case) between financial and insurance risks.}
	\label{fig:Investment-reinsurance strategies}
\end{figure}

In Figure \ref{fig:Investment-reinsurance strategies} we display the time-$0$ robust optimal investment (ROIS) and reinsurance (RORS) strategies with respect to the risk aversion for all three cases: the ambiguity-neutral case (no uncertainty), the no-correlation between financial and insurance risks case, and the general case which is determined through the first equation in \eqref{Optimal strategy_AAI}. We observe that, except in the no-correlation case (see Figure \ref{fig:Investment-reinsurance vs gamma without uncertainty}), the ROIS always dominates the RORS (see Figures \ref{fig:Investment-reinsurance vs gamma under uncertainty} and \ref{fig:Investment-reinsurance vs gamma without correlation}). Our numerical results show similar behaviours of the ROIS and the RORS with respect to the EIS coefficient. In addition, all three graphs show a monotonic decline (with different magnitude) of ROIS and RORS as risk aversion increases. On contrary, our numerical results show that the EIS has little effect on the ROIS and the RORS.
\begin{figure}[h!]
	\centering
	\subfigure[]{
		\includegraphics[width=0.45\textwidth]{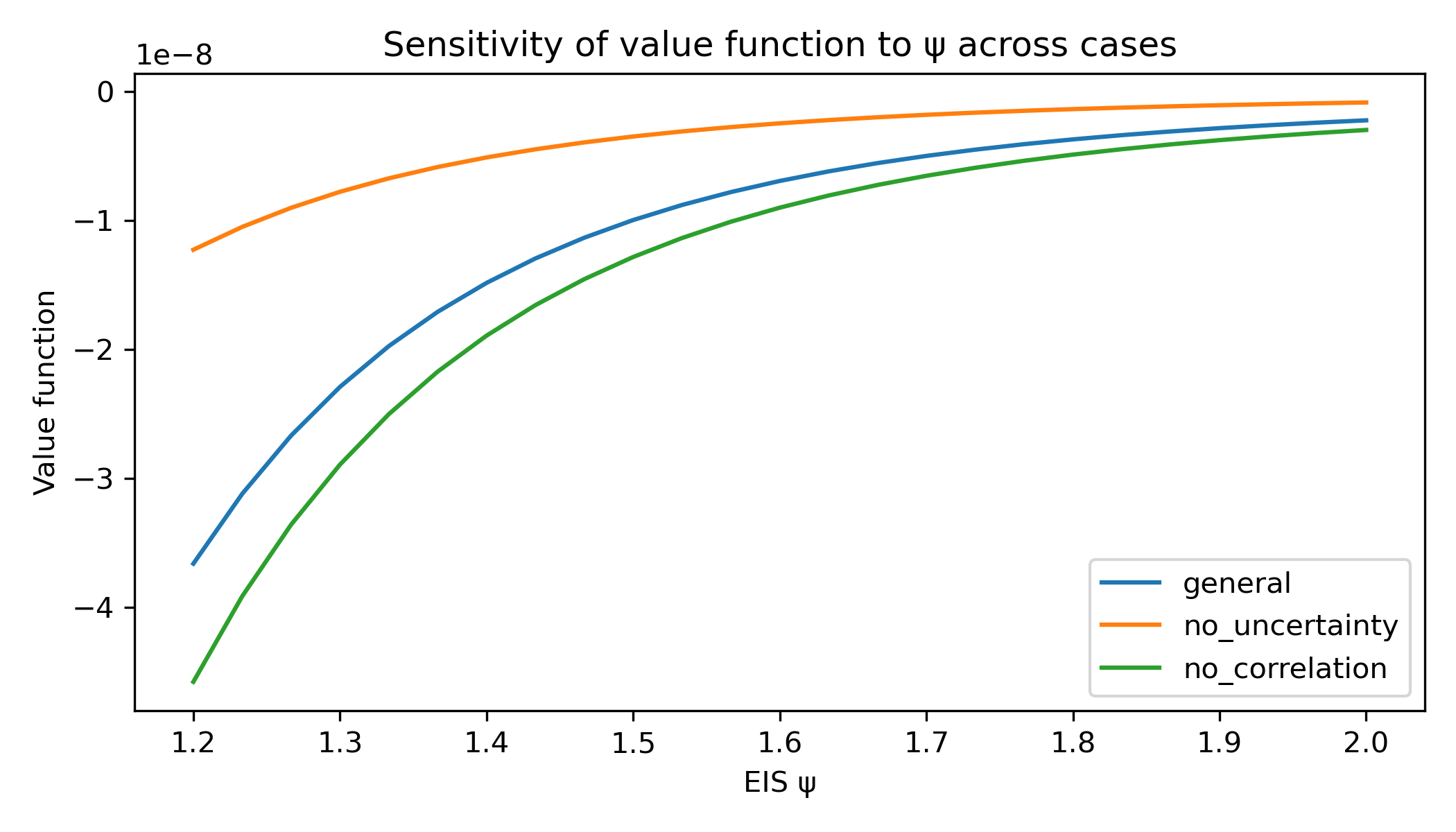}
		\label{fig:Value function vs EIS}
	}
	\subfigure[]{
		\includegraphics[width=0.45\textwidth]{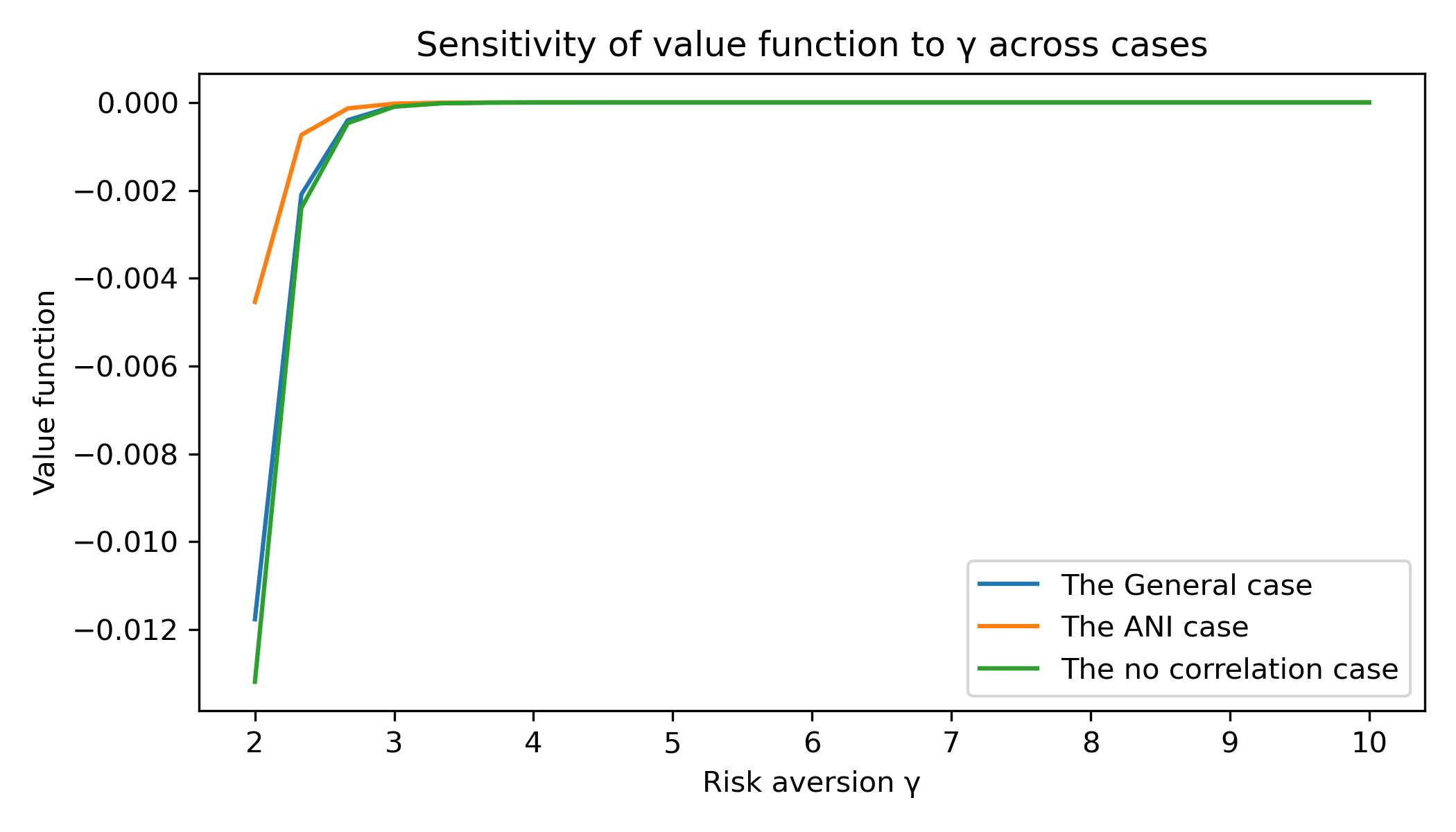}
		\label{fig:Value function vs gamma}
	}
	\vfill
	\subfigure[]{
		\includegraphics[width=0.45\textwidth]{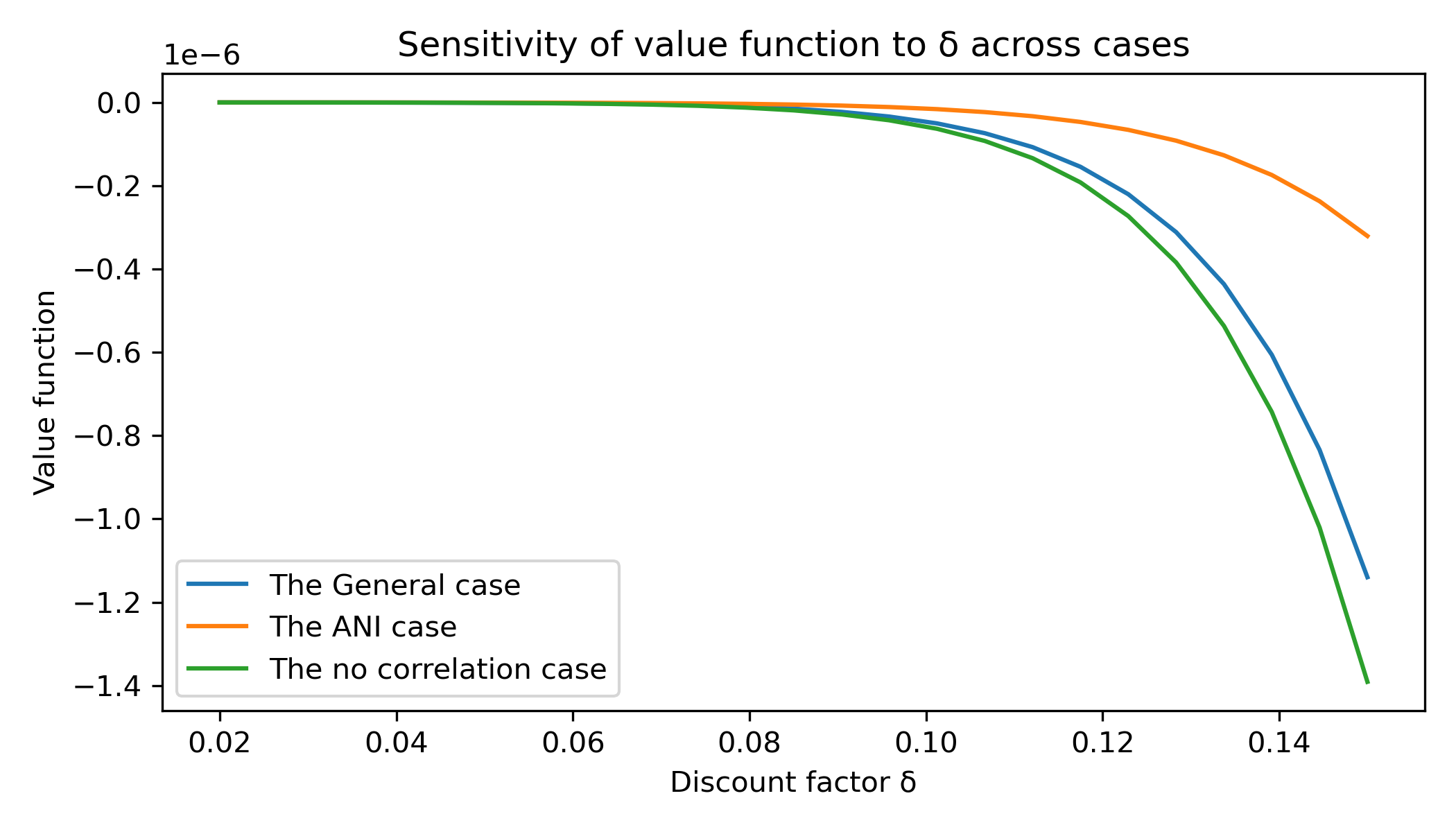}
		\label{fig:Value function vs delta}
	}
	\caption{The value function  for all cases.}
	\label{fig:Value functions}
\end{figure}

In Figure \ref{fig:Value functions} shows the sensitivity of the value function with respect to the EIS coefficient (see Figure \ref{fig:Value function vs EIS}), the risk aversion coefficient (see Figure \ref{fig:Value function vs gamma}) and the discount factor (see Figure \ref{fig:Value function vs delta}) for all the three cases mentioned in the previous paragraph.

\section{Conclusion}\label{Conclusion}
This paper addresses the complex problem of how an ambiguity-averse insurer should optimally manage consumption, investment, and reinsurance over a finite time horizon. The insurer's wealth dynamics incorporate a financial market (a risk-free bond and a risky asset) and an insurance surplus process based on the diffusion approximation of the classical Cramér-Lundberg model. A key challenge is that the insurer operates under model uncertainty (ambiguity) regarding the true probabilities of asset returns and insurance claims. Furthermore, the insurer's preferences are modeled using Epstein-Zin recursive utility, which allows for a separation between risk aversion and the elasticity of intertemporal substitution (EIS), a more realistic and flexible framework than traditional time-additive utilities.

To solve this robust control problem, a max-min optimisation problem is formulated, where the insurer maximises utility under the worst-case scenario from a set of plausible models, penalised by relative entropy. The solution is achieved by characterising the problem through a system of coupled forward-backward stochastic differential equations (FBSDEs). Using the martingale optimality principle, a closed-form analytical expressions for the optimal consumption is derived, investment, reinsurance, and the corresponding worst-case distortion process studied.

Through simulation, the results observed yield several important insights. The explicit formulas show that the optimal reinsurance strategy depends on financial market parameters, and the investment strategy depends on insurance market parameters, demonstrating an intrinsic co-dependence even when the two markets are uncorrelated. Numerical analyses confirm that optimal consumption decreases with higher risk aversion and EIS, while both investment and reinsurance strategies monotonically decline as risk aversion increases. The study successfully integrates robustness, recursive preferences, and liability management into a unified framework, providing actionable strategies for insurers navigating deep uncertainty.

\appendix
\section{Proof of Proposition~\ref{Non-empty control set}}\label{Non-empty control set_Proof}
\begin{proof}
	We construct $V^{c,\xi}$, given by \eqref{Epstein-Zin utility_Maenhout's style}, via the BSDE
	\begin{align}\label{Epstein-Zin utility_Maenhout's style_Integral form0}
		V_{t}^{c,\xi}&=h(c_{T})+\int_{t}^{T}\Big(f(c_{s},V_{s}^{c,\xi})+\frac{1}{2\Phi}\|\xi_{s}\|^{2}(1-\gamma)V_{s}^{c,\xi}\Big)\mathrm{d}s-\int_{t}^{T}Z_{t}^{c,\xi}\mathrm{d}B_{s}^{\mathbb{Q}^{\xi}}.
	\end{align}
	Recall the definition of $f$ in \eqref{Epstein-Zin generator} with $\gamma,\psi>1$ (that is, $\theta<0$). Then the generator of the BSDE \eqref{Epstein-Zin utility_Maenhout's style_Integral form0} is not Lipschitz. We obtain the unique solution of \eqref{Epstein-Zin utility_Maenhout's style_Integral form0} in a suitable space via the transformation
	\begin{align*}
		\big(Y_{t},Z_{t}\big):=e^{\int_{0}^{t}\frac{1}{2\Phi}\|\xi_{s}\|^{2}\mathrm{d}s}(1-\gamma)\big(V_{t}^{c,\xi},Z_{t}^{c,\xi}\big),~t\in[0,T],
	\end{align*}
	so that Equation \eqref{Epstein-Zin utility_Maenhout's style_Integral form0} becomes
	\begin{align}\label{Epstein-Zin utility_Maenhout's style_Integral form_Transformed}
		Y_{t}&=e^{-\delta\theta T}\big(e^{\int_{0}^{T}\frac{1}{2\Phi(1-\gamma)} \|\xi_{s}\|^{2}\mathrm{d}s}c_{T}\big)^{1-\gamma}\notag\\
		&\phantom{X}+\int_{t}^{T}\delta\theta e^{-\delta s}\big(e^{\int_{0}^{s}\frac{1}{2\Phi(1-\gamma)}\|\xi_{u}\|^{2} \mathrm{d}u}c_{s}\big)^{1-\frac{1}{\psi}}Y_{s}^{1-\frac{1}{\theta}}\mathrm{d}s-\int_{t}^{T}Z_{t}\mathrm{d}B_{s}^{\mathbb{Q}^{\xi}}.
	\end{align}
	This is precisely the type of BSDE considered in \cite[Prop.~2.2]{xing2017consumption} with $c_{s}$ replaced by $e^{\int_{0}^{s}\frac{1}{2\Phi(1-\gamma)}\|\xi_{u}\|^{2}\mathrm{d}u}c_{s}$ for $0\le s\le T$. Hence, by the proof of \cite[Prop.~2.2]{xing2017consumption}, the unique solution $(Y,Z)$ of the BSDE \eqref{Epstein-Zin utility_Maenhout's style_Integral form_Transformed} is such that $Y$ is continuous, strictly positive and belongs to the class (D), and $\int_{0}^{T}\|Z_{s}\|^{2}\mathrm{d}s<\infty$ $\mathbb{Q}^{\xi}$-a.s. Using the fact that $V_{t}^{c,\xi}=\frac{1}{1-\gamma}e^{-\int_{0}^{t}\frac{1}{2\Phi}\|\xi_{s}\|^{2}\mathrm{d}s}Y_{t}$ for $t\in[0,T]$, with $t\mapsto\frac{1}{1-\gamma}e^{-\int_{0}^{t}\frac{1}{2\Phi}\|\xi_{s}\|^{2}\mathrm{d}s}$ bounded almost surely, we deduce that the process $V^{c,\xi}$ is continuous, strictly negative and of class (D). Moreover, using the fact that $\Phi\ge0$, we have
	\begin{align*}
		\int_{0}^{T}\|Z_{s}^{c,\xi}\|^{2}\mathrm{d}s&=\frac{1}{(1-\gamma)^{2}}\int_{0}^{T}e^{-\int_{0}^{s}\frac{1}{\Phi}\|\xi_{u}\|^{2}\mathrm{d}u}\|Z_{s}\|^{2}\mathrm{d}s\\
		&<\frac{1}{(1-\gamma)^{2}}\int_{0}^{T}\|Z_{s}\|^{2}\mathrm{d}s<\infty.
	\end{align*}
	Hence, $\int_{0}^{T}\|Z_{s}^{c,\xi}\|^{2}\mathrm{d}s<\infty$ $\mathbb{Q}^{\xi}$-a.s.
	That concludes the proof.
\end{proof}

\section{Proof of Proposition~\ref{Existence result on the FBSDE}}\label{Solution to the FBSDE_Proof}
\begin{proof}
	We show that the triple $(\widehat{X},Y,Z)$ given by \eqref{Solution to the FBSDE_AAI} satisfies the FBSDE \eqref{Optimal FBSDE}. Let $\bar{x}$ denotes the constant defined by $\bar{x}:=\widetilde{x}\big(-r_{m}-\delta^{\psi}+\frac{\Phi}{(\gamma+\Phi)^{2}}\|\eta\|^{2}\big)$. Applying It{\^o}'s formula to $Y$ we obtain
	\begin{align}\label{E1_Appendix B}
		\mathrm{d}Y_{t}&=-\bar{x}e^{r_{m}t}\exp\Big(\frac{-1+2(\gamma+\Phi)}{2(\gamma+\Phi)^{2}}\|\eta\|^{2}t+\frac{1}{\gamma+\Phi}\eta^{\intercal}B_{t}\Big)\mathrm{d}t\notag\\
		&+\bar{x}\frac{e^{r_{m}T}-e^{r_{m}t}}{r_{m}}\big(\frac{1}{\gamma+\Phi}\|\eta\|^{2}\mathrm{d}t+\frac{1}{\gamma+\Phi}\eta^{\intercal}\mathrm{d}B_{t}\big)\exp\Big(\frac{-1+2(\gamma+\Phi)}{2(\gamma+\Phi)^{2}}\|\eta\|^{2}t+\frac{1}{\gamma+\Phi}\eta^{\intercal}B_{t}\Big)\notag\\
		&=\Big(-\bar{x}e^{r_{m}t}+\bar{x}\frac{e^{r_{m}T}-e^{r_{m}t}}{r_{m}}\frac{1}{\gamma+\Phi}\|\eta\|^{2}\Big)\exp\Big(\frac{-1+2(\gamma+\Phi)}{2(\gamma+\Phi)^{2}}\|\eta\|^{2}t+\frac{1}{\gamma+\Phi}\eta^{\intercal}B_{t}\Big)\mathrm{d}t\notag\\
		&+\bar{x}\frac{e^{r_{m}T}-e^{r_{m}t}}{r_{m}}\exp\Big(\frac{-1+2(\gamma+\Phi)}{2(\gamma+\Phi)^{2}}\|\eta\|^{2}t+\frac{1}{\gamma+\Phi}\eta^{\intercal}B_{t}\Big)\frac{1}{\gamma+\Phi}\eta^{\intercal}\mathrm{d}B_{t}.
	\end{align}
	Using the definition of $Y$ in \eqref{Solution to the FBSDE_AAI} we deduce that
	\begin{align*}
		&\bar{x}\frac{e^{r_{m}T}-e^{r_{m}t}}{r_{m}}\exp\Big(\frac{-1+2(\gamma+\Phi)}{2(\gamma+\Phi)^{2}}\|\eta\|^{2}t+\frac{1}{\gamma+\Phi}\eta^{\intercal}B_{t}\Big)\frac{1}{\gamma+\Phi}\eta\\
		&=\frac{1}{\gamma+\Phi}\big(Y_{t}+e^{-rT}G\big)\eta.
	\end{align*}
	Let $Z_{t}=\frac{1}{\gamma+\Phi}\big(Y_{t}+e^{-rT}G\big)\eta$ for $t\in[0,T]$. Then the generator of the BSDE \eqref{E1_Appendix B} becomes
	\begin{align*}
		Z_{t}^{\intercal}\eta-\bar{x}e^{r_{m}t}\exp\Big(\frac{-1+2(\gamma+\Phi)}{2(\gamma+\Phi)^{2}}\|\eta\|^{2}t+\frac{1}{\gamma+\Phi}\eta^{\intercal}B_{t}\Big).
	\end{align*}
	Hence, using the definition of $\widetilde{x},\widehat{\xi}$ and $\widehat{X}_{t}+e^{rt}Y_{t},~0\le t\le T$, we have
	\begin{align*}
		\mathrm{d}Y_{t}&=-\Big(e^{-rt}\big(\frac{\delta^{\psi}}{\psi-1}+r+\frac{1}{2}\frac{1}{\gamma+\Phi}\|\eta\|^{2}{-\frac{\delta\theta}{1-\gamma}}\big)\big(\widehat{X}_{t}+e^{rt}Y_{t}\big)-Z_{t}^{\intercal}\eta\Big)\mathrm{d}t+Z_{t}^{\intercal}\mathrm{d}B_{t}\\
		&=-\Big(e^{-rt}\big(\frac{\delta^{\psi}}{\psi-1}+r+\frac{1}{2}\frac{1}{\gamma+\Phi}\|\eta\|^{2}{-\frac{\delta\theta}{1-\gamma}}\big)\big(\widehat{X}_{t}+e^{rt}Y_{t}\big)-Z_{t}^{\intercal}\eta\Big)\mathrm{d}t\\
		&\phantom{X}+Z_{t}^{\intercal}\big(\mathrm{d}B_{t}^{\mathbb{Q}^{\widehat{\xi}}}-\frac{\Phi}{\gamma+\Phi}\eta\mathrm{d}t\big)\\
		&=-\Big(e^{-rt}\big(\frac{\delta^{\psi}}{\psi-1}+r+\frac{1}{2}\frac{1}{\gamma+\Phi}\|\eta\|^{2}{-\frac{\delta\theta}{1-\gamma}}\big)\big(\widehat{X}_{t}+e^{rt}Y_{t}\big)-\frac{\gamma}{\gamma+\Phi}Z_{t}^{\intercal}\eta\Big)\mathrm{d}t\\
		&\phantom{X}+Z_{t}^{\intercal}\mathrm{d}B_{t}^{\mathbb{Q}^{\widehat{\xi}}}.
	\end{align*}
	Similar arguments applied to $\widehat{X}$ give
	\begin{align*}
		\mathrm{d}\widehat{X}_{t}&=\Big(r\widehat{X}_{t}+\big(-\delta^{\psi}{+\frac{1}{\gamma+\Phi}}\|\eta\|^{2}\big)(\widehat{X}_{t}+e^{rt}Y_{t})-{\frac{\gamma}{\gamma+\Phi}}e^{rt}Z_{t}^{\intercal}\eta\Big)\mathrm{d}t\\
		&\phantom{X}+\Big({\frac{1}{\gamma+\Phi}}(\widehat{X}_{t}+e^{rt}Y_{t})\eta^{\intercal}-e^{rt}Z_{t}^{\intercal}\Big)\mathrm{d}B_{t}^{\mathbb{Q}^{\widehat{\xi}}}.
	\end{align*}
	
	Local uniqueness follows from lemma $2.1$ in \cite{xie2020exploration}. That concludes the proof.
\end{proof}

\section{Proof of Lemmas \ref{Admissiblity of the candidate controls}, \ref{Comparaison theorem} and \ref{Upper bound_Value function}, and Theorem \ref{Mainresult_AAI}}\label{Mainresult_AAI_Proof}
\begin{proof}[Proof of Lemma \ref{Admissiblity of the candidate controls}]
	The proof is split in three steps.
	
	\textit{Step 1:} (The positivity of $\widehat{X}_{t}+e^{rt}Y_{t}>0$ for $t\in[0,T]$). Since $\widehat{X}_{t}+e^{rt}Y_{t}=\widetilde{x}\varphi_{t}$, the proof follows directly from the first and third conditions in Assumption \ref{Well-definedness of the reinsurance strategy}, and the positivity of $\varphi$ defined in \eqref{Auxiliary process_AAI}.
	
	\textit{Step 2:} (The class (D) property of positivity of $(\widehat{X}+e^{rt}Y)^{1-\gamma}$). We have
	\begin{align}\label{Class (D) property}
		&(\widehat{X}_{t}+e^{rt}Y_{t})^{1-\gamma}\notag\\
		&=\widetilde{x}^{1-\gamma}\exp\Big(\Big(-\delta^{\psi}\theta+{\frac{(1-\gamma)(\gamma+3\Phi-1)}{2(\gamma+\Phi)^{2}}}\|\eta\|^{2}+\delta\theta\Big)t+\frac{1-\gamma}{\gamma+\Phi}\eta^{\intercal}B_{t}\Big)\notag\\
		&=\widetilde{x}^{1-\gamma}\exp\Big(\Big(-\delta^{\psi}\theta+{\frac{3\Phi(1-\gamma)}{2(\gamma+\Phi)^{2}}}\|\eta\|^{2}+\delta\theta\Big)t\Big)\mathcal{E}\big(\int_{}^{}\frac{1-\gamma}{\gamma+\Phi}\eta^{\intercal}\mathrm{d}B\big)_{t},
	\end{align}
	where $\mathcal{E}\big(\int_{}^{}\beta_{s}\mathrm{d}B_{s}\big)_{t}:=\exp\big(-\frac{1}{2}\int_{0}^{t}\|\beta_{s}\|^{2}\mathrm{d}s+\int_{0}^{t}\beta_{s}\mathrm{d}B_{s}\big)$ is the Dol{\'e}ans-Dade exponential at time $t$. Observe that the process $\mathcal{E}\big(\int_{}^{}\frac{1-\gamma}{\gamma+\Phi}\eta^{\intercal}\mathrm{d}B\big)$ is a $\mathbb{P}$-martingale (hence of class (D)); because $\frac{1-\gamma}{\gamma+\Phi}\eta^{\intercal}\in\mathbb{R}^{2}$. Hence the right-side of \eqref{Class (D) property} is of class (D) as a product of a bounded deterministic function (because the constant $\widetilde{x}$ is positive and finite) and a process of class (D). Thus, $(\widehat{X}+e^{rt}Y)^{1-\gamma}$ is of class (D).
	
	\textit{Step 3:} (Confirm that $(\widehat{c},\widehat{\xi})\in\mathcal{A}_{a}$). Recall from \eqref{Candidate distortion process} and \eqref{Candidate consumption_AAI} that $\widehat{\xi}_{t}=\frac{\Phi}{\gamma+\Phi}\eta$ (meaning, $\widehat{\xi}$ is a constant) and $\widehat{c}_{t}=\delta^{\psi}(\widehat{X}_{t}+e^{rt}Y_{t})$ for $t\in[0,T]$. Then, using the definition of $\varphi$ in \eqref{Auxiliary process_AAI}, Girsanov theorem and the facts that $c_{T}=\widehat{X}_{T}+e^{rT}Y_{T}$ (see Definition \ref{Admissible strategies}) and $\widehat{\xi}$ is a constant,  we obtain
	\begin{align*}
		&\mathbb{E}^{\mathbb{Q}^{\widehat{\xi}}}\Big[e^{\int_{0}^{T}\frac{1}{2\Phi} \|\widehat{\xi}_{s}\|^{2}\mathrm{d}s}c_{T}^{1-\gamma}\Big]\\
		&=\mathbb{E}\Big[\delta^{\psi(1-\gamma)}e^{\frac{\Phi}{2(\gamma+\Phi)^{2}}\|\eta\|^{2}T}\mathcal{E}\Big(\int_{}^{}-\frac{\Phi}{\gamma+\Phi}\eta^{\intercal}\mathrm{d}B\Big)_{T}\notag\\
		&\phantom{XX}\times\exp\Big(\Big(-\delta^{\psi}\theta+{\frac{(1-\gamma)(\gamma+3\Phi-1)}{2(\gamma+\Phi)^{2}}}\|\eta\|^{2}+\delta\theta\Big)T+\frac{1-\gamma}{\gamma+\Phi}\eta^{\intercal}B_{T}\Big)\Big]\notag\\
		&=\mathbb{E}\Big[\delta^{\psi(1-\gamma)}e^{\frac{\Phi}{2(\gamma+\Phi)^{2}}\|\eta\|^{2}T}\exp\Big(\Big(-\delta^{\psi}\theta+{\frac{\Phi(1-\gamma)}{2(\gamma+\Phi)^{2}}}\|\eta\|^{2}+\delta\theta\Big)T\Big)\notag\\
		&\phantom{XX}\times\mathcal{E}\Big(\int_{}^{}\frac{1-\gamma-\Phi}{\gamma+\Phi}\eta^{\intercal}\mathrm{d}B\Big)_{T}\Big]\notag\\
		&=\delta^{\psi(1-\gamma)}e^{\frac{\Phi}{2(\gamma+\Phi)^{2}}\|\eta\|^{2}T}\exp\Big(\Big(-\delta^{\psi}\theta+{\frac{\Phi(1-\gamma)}{2(\gamma+\Phi)^{2}}}\|\eta\|^{2}+\delta\theta\Big)T\Big)\notag\\
		&<\infty,
	\end{align*}
	where the third equality holds due to $\mathcal{E}\Big(\int_{}^{}\frac{1-\gamma-\Phi}{\gamma+\Phi}\eta^{\intercal}\mathrm{d}B\Big)$ being a $\mathbb{P}$-martingale. Besides, we have
	\begin{align*}
		&\mathbb{E}^{\mathbb{Q}^{\xi}}\Big[\int_{0}^{T}e^{-\delta s}c_{s}^{1-\frac{1}{\psi}}\mathrm{d}s\Big]\notag\\
		&=\delta^{\psi-1}\widetilde{x}^{1-\frac{1}{\psi}}\mathbb{E}\Big[\mathcal{E}\Big(\int_{}^{}-\frac{\Phi}{\gamma+\Phi}\eta^{\intercal}\mathrm{d}B\Big)_{T}\notag\\
		&\times\int_{0}^{T}e^{-\delta s}\exp\Big(\Big(-\delta^{\psi}+\frac{\psi}{\psi-1}\frac{\gamma+3\Phi-1}{2(\gamma+\Phi)^{2}}\|\eta\|^{2}+\delta\Big)s+\frac{\psi}{\psi-1}\frac{1}{\gamma+\Phi}\eta^{\intercal}B_{s}\Big)\mathrm{d}s\Big]\notag\\
		&\le\delta^{\psi-1}\widetilde{x}^{1-\frac{1}{\psi}}\Big(\mathbb{E}\Big[\mathcal{E}\Big(\int_{}^{}-\frac{\Phi}{\gamma+\Phi}\eta^{\intercal}\mathrm{d}B\Big)_{T}^{2}\Big]\Big)^{\frac{1}{2}}\notag\\
		&\times\Big(\mathbb{E}\Big[\Big(\int_{0}^{T}\exp\Big(\big(-\delta^{\psi}+\frac{\psi}{\psi-1}\frac{\gamma+3\Phi-1}{2(\gamma+\Phi)^{2}}\|\eta\|^{2}\big)s+\frac{\psi}{\psi-1}\frac{1}{\gamma+\Phi}\eta^{\intercal}B_{s}\Big)\mathrm{d}s\Big)^{2}\Big]\Big)^{\frac{1}{2}}\notag\\
		&\le\delta^{\psi-1}\widetilde{x}^{1-\frac{1}{\psi}}\Big(\mathbb{E}\Big[\exp\Big(-\Big(\frac{\Phi}{\gamma+\Phi}\Big)^{2}\|\eta\|^{2}T-\frac{2\Phi}{\gamma+\Phi}\eta^{\intercal}B_{T}\Big)\Big]\Big)^{\frac{1}{2}}\notag\\
		&\times\Big(T^{2}\mathbb{E}\Big[\int_{0}^{T}\exp\big(\Big(-2\delta^{\psi}+\frac{\psi}{\psi-1}\frac{\gamma+3\Phi-1}{(\gamma+\Phi)^{2}}\|\eta\|^{2}\Big)s+\frac{2\psi}{\psi-1}\frac{1}{\gamma+\Phi}\eta^{\intercal}B_{s}\big)\mathrm{d}s\Big]\Big)^{\frac{1}{2}}\notag\\
		&=\delta^{\psi-1}\widetilde{x}^{1-\frac{1}{\psi}}\exp\Big(\frac{\Phi^{2}}{2(\gamma+\Phi)^{2}}\|\eta\|^{2}T\Big)\Big(\mathbb{E}\Big[\mathcal{E}\Big(\int_{}^{}-\frac{2\Phi}{\gamma+\Phi}\eta^{\intercal}\mathrm{d}B\Big)_{T}\Big]\Big)^{\frac{1}{2}}\notag\\
		&\times\Big(\mathbb{E}\Big[\int_{0}^{T}\exp\Big(\Big(-2\delta^{\psi}+\frac{\psi}{(\psi-1)^{2}}\frac{(\psi-1)(\gamma+3\Phi-1)+2\psi}{(\gamma+\Phi)^{2}}\|\eta\|^{2}+\delta\Big)s\Big)\notag\\
		&\phantom{XXXXX}\times\mathcal{E}\Big(\int_{}^{}\frac{2\psi}{\psi-1}\frac{1}{\gamma+\Phi}\eta^{\intercal}\mathrm{d}B\Big)_{s}\mathrm{d}s\Big]\Big)^{\frac{1}{2}}\notag\\
		&=\delta^{\psi-1}\widetilde{x}^{1-\frac{1}{\psi}}\exp\Big(\frac{\Phi^{2}}{2(\gamma+\Phi)^{2}}\|\eta\|^{2}T\Big)\notag\\
		&\times\int_{0}^{T}\exp\Big(\Big(-2\delta^{\psi}+\frac{\psi}{(\psi-1)^{2}}\frac{(\psi-1)(\gamma+3\Phi-1)+2\psi}{(\gamma+\Phi)^{2}}\|\eta\|^{2}+\delta\Big)s\Big)\mathrm{d}s\notag\\
		&<\infty,
	\end{align*}
	where the first inequality follows from Cauchy-Schwarz's inequality, the second comes from Jensen's inequality, and the third equality holds due to the fact that $\mathcal{E}\Big(\int_{}^{}-\frac{2\Phi}{\gamma+\Phi}\eta^{\intercal}\mathrm{d}B\Big)$ and $\mathcal{E}\Big(\int_{}^{}\frac{2\psi}{\psi-1}\frac{1}{\gamma+\Phi}\eta^{\intercal}\mathrm{d}B\Big)$ are $\mathbb{P}$-martingales.
\end{proof}
\begin{proof}[Proof of Lemma \ref{Comparaison theorem}]
	The proof follows similar arguments as in the third step in the proof of proposition $2.2$ in \cite{xing2017consumption} with the generators $F(s,c_{s},Y_{s})$ and $F(s,c_{s},\tilde{Y}_{s})$ replaced by $f(c_{s},Y_{s})+\frac{1}{2\Phi}\|\xi_{s}\|^{2}(1-\gamma)Y_{s}$ and $f(c_{s},\tilde{Y}_{s})+\frac{1}{2\Phi}\|\xi_{s}\|^{2}(1-\gamma)\tilde{Y}_{s}$, respectively, for all $s\in[0,T]$.
\end{proof}
\begin{proof}[Proof of Lemma \ref{Upper bound_Value function}]
	For a triple $(c,\pi,\xi)$ of admissible consumption, investment-reinsurance and distortion strategies (that is, $(c,\pi,\xi)\in\mathcal{A}^{AAI}$; see Definition~\ref{Admissible strategies}). Let $(M_{t})_{t\in[0,T]}$ be the process given in \eqref{Optimal martingale} by
	\begin{align*}
		M_{t}^{c,\pi,\xi}&:={e^{-\delta\theta t}}\frac{(X_{t}+e^{rt}Y_{t})^{1-\gamma}}{1-\gamma}\notag\\
		&\phantom{xx}+\int_{0}^{t}\Big(f\big(c_{s},e^{-\delta\theta s}\frac{(X_{s}+e^{rs}Y_{s})^{1-\gamma}}{1-\gamma}\big)+\frac{1}{2\Phi}\|\xi\|^{2}(X_{s}+e^{rs}Y_{s})^{1-\gamma}\Big)\mathrm{d}s.
	\end{align*}
	Using \eqref{OptEZ_Differential_AAI} and \eqref{YDrift_AAI} we deduce that $M$ is a local super-martingale. Moreover, using the Doob-Meyer decomposition and martingale representation, there exists an increasing process $A$  and a process $Z^{M}$ such that $M=-A+\int_{0}^{\cdot}Z_{s}^{M}\mathrm{d}B_{s}^{\mathbb{Q}^{\xi}}$. Hence, $(e^{-\delta\theta\cdot}\frac{(X+e^{\int_{0}^{}r_{s}\mathrm{d}s}Y)^{1-\gamma}}{1-\gamma},Z^{M})$ is a super-solution to \eqref{Epstein-Zin utility_Maenhout's style_Integral form} with integrable terminal condition $e^{-\delta\theta T}\frac{(X_{T}-G)^{1-\gamma}}{1-\gamma}$; see Lemma \ref{Comparaison theorem} for the notion of sub-/super- solutions of BSDEs. Now, consider the utility $V^{c,\xi}$ associated to the consumption stream $c$ and the terminal lump sum $X_{T}-G$; meaning that $V^{c,\xi}$ is the first part of the solution of the BSDE \eqref{Epstein-Zin utility_Maenhout's style_Integral form} with terminal value $h(X_{T}-G)$. Therefore, using Lemma \ref{Comparaison theorem} we confirm Equation \eqref{Upper bound_Value function1}.
\end{proof}
\begin{proof}[Proof of lemma \ref{Auxiliary martingale}]
	Consider the process $M^{c,\pi,\xi}$ defined by \eqref{OptEZ_Differential_AAI} for all $(c,\pi,\xi)\in\mathcal{A}^{AAI}$. For the consumption $\widehat{c}$, investment-reinsurance $\widehat{\pi}$ and distortion process $\widehat{\xi}$ (with associated function $\mathcal{H}$, given by \eqref{Auxiliary BSDE generator_AAI}, of the BSDE \eqref{Auxiliary BSDE_AAI}) given by \eqref{Candidate consumption_AAI}, \eqref{Candidate investment_AAI}, \eqref{Candidate distortion process} and , respectively, one can show that
	\begin{align}\label{Auxiliary martingale process_SDE1}
		\mathrm{d}M_{t}^{\widehat{c},\widehat{\pi},\widehat{\xi}}&=e^{-\delta\theta t}\frac{(\widehat{X}_{t}^{F}+e^{\int_{0}^{t}r_{s}\mathrm{d}s}Y_{t}^{F})^{1-\gamma}}{1-\gamma}\frac{1-\gamma}{\gamma+\Phi}\eta^{\intercal}\mathrm{d}B_{t}^{\mathbb{Q}^{\xi}},~0\le t\le T.
	\end{align}
	On the other hand, using successively \eqref{Optimal martingale} and \eqref{Epstein-Zin generator} we have
	\begin{align}\label{Auxiliary martingale process_SDE2}
		\mathrm{d}M_{t}^{\widehat{c},\widehat{\pi},\widehat{\xi}}&=\mathrm{d}\Big({e^{-\delta\theta t}}\frac{(\widehat{X}_{t}^{F}+e^{\int_{0}^{t}r_{s}\mathrm{d}s}Y_{t}^{F})^{1-\gamma}}{1-\gamma}\Big)\notag\\
		&\phantom{X}+\Big(f\big(\widehat{c}_{t},e^{-\delta\theta t}\frac{(\widehat{X}_{t}+e^{rt}Y_{t})^{1-\gamma}}{1-\gamma}\big)+\frac{\Phi}{2(\gamma+\Phi)^{2}}\|\eta\|^{2}(\widehat{X}_{t}+e^{rt}Y_{t})^{1-\gamma}\Big)\mathrm{d}t\nonumber\\
		&=\mathrm{d}\Big(e^{-\delta\theta t}\frac{(\widehat{X}_{t}^{F}+e^{\int_{0}^{t}r_{s}\mathrm{d}s}Y_{t}^{F})^{1-\gamma}}{1-\gamma}\Big)\notag\\
		&\phantom{X}+\Big(\delta^{\psi}\theta+\frac{\Phi(1-\gamma)}{2(\gamma+\Phi)^{2}}\|\eta\|^{2}e^{\delta\theta t}\Big)e^{-\delta\theta t} \frac{(\widehat{X}_{t}^{F}+e^{\int_{0}^{t}r_{s}\mathrm{d}s}Y_{t}^{F})^{1-\gamma}}{1-\gamma}\mathrm{d}t.
	\end{align}
	Hence, combining \eqref{Auxiliary martingale process_SDE1} and \eqref{Auxiliary martingale process_SDE2} we obtain
	\begin{align}\label{Auxiliary martingale process_SDE3}
		&\mathrm{d}\Big(e^{-\delta\theta t}\frac{(\widehat{X}_{t}^{F}+e^{\int_{0}^{t}r_{s}\mathrm{d}s}Y_{t}^{F})^{1-\gamma}}{1-\gamma}\Big)+\Big(\delta^{\psi}\theta+\frac{\Phi(1-\gamma)}{2(\gamma+\Phi)^{2}}\|\eta\|^{2}e^{\delta\theta t}\Big)e^{-\delta\theta t} \frac{(\widehat{X}_{t}^{F}+e^{\int_{0}^{t}r_{s}\mathrm{d}s}Y_{t}^{F})^{1-\gamma}}{1-\gamma}\mathrm{d}t\nonumber\\
		&={e^{-\delta\theta t}}\frac{(\widehat{X}_{t}^{F}+e^{\int_{0}^{t}r_{s}\mathrm{d}s}Y_{t}^{F})^{1-\gamma}}{1-\gamma}\frac{1-\gamma}{\gamma}\eta_{t}^{\intercal}\mathrm{d}B_{t}^{\mathbb{Q}^{\xi}}.
	\end{align}
	Multiplying both sides of \eqref{Auxiliary martingale process_SDE3} by $\exp\big(\int_{0}^{t}\big(\delta^{\psi}\theta+\frac{\Phi(1-\gamma)}{2(\gamma+\Phi)^{2}}\|\eta\|^{2}e^{\delta\theta s}\big)\mathrm{d}s\big),~0\le t\le T$, we have
	\begin{align}
		\widetilde{M}_{t}=\widetilde{M}_{0}\mathcal{E}\big(\int\frac{1-\gamma}{\gamma}\eta^{\intercal}\mathrm{d}B^{\mathbb{Q}^{\xi}}\big)_{t}~\text{ for }t\in[0,T],
	\end{align}
	where $\mathcal{E}(\int\beta^{\intercal}\mathrm{d}B^{\mathbb{Q}^{\xi}})_{t}:=\exp\left(-\frac{1}{2}\int_{0}^{t}\|\beta_{s}\|^{2}\mathrm{d}s+\int_{0}^{t}\beta_{s}^{\intercal}\mathrm{d}B_{s}^{\mathbb{Q}^{\xi}}\right)$.
\end{proof}
\begin{proof}[Proof of Theorem \ref{Mainresult_AAI}]
	Thanks to Lemma \ref{Admissiblity of the candidate controls} the uplet $(\widehat{c},\widehat{\pi}^{S},\widehat{\pi}^{re},\widehat{\xi})$ given by \eqref{Optimal strategy_AAI} is admissible in the sense of Definition~\ref{Admissible strategies} with $\widehat{\pi}=(\widehat{\pi}^{S},\widehat{\pi}^{re})\Sigma$. Next, we prove that $(\widehat{c},\widehat{\pi}^{S},\widehat{\pi}^{re},\widehat{\xi})$ is optimal. Let $\widetilde{M}$ be as in Lemma \ref{Auxiliary martingale}. Thanks to Lemma \ref{Auxiliary martingale}, there exists a square integrable process $\widetilde{Z}$ such that
	\begin{align}\label{Auxiliary martingale process_SDE4}
		\mathrm{d}\widetilde{M}_{t}&=\widetilde{Z}_{t}\mathrm{d}B_{t}^{\mathbb{Q}^{\widehat{\xi}}},~0\le t\le T.
	\end{align}
	Substituting \eqref{Auxiliary martingale process} into the left-side of \eqref{Auxiliary martingale process_SDE4} and applying It{\^o}'s formula we obtain
	\begin{align*}
		&\mathrm{d}\Big(e^{-\delta\theta t}\frac{(\widehat{X}_{t}^{F}+e^{\int_{0}^{t}r_{s}\mathrm{d}s}Y_{t}^{F})^{1-\gamma}}{1-\gamma}\Big)+\Big(\delta^{\psi}\theta+\frac{\Phi(1-\gamma)}{2(\gamma+\Phi)^{2}}\|\eta\|^{2}e^{\delta\theta t}\Big)e^{-\delta\theta t} \frac{(\widehat{X}_{t}^{F}+e^{\int_{0}^{t}r_{s}\mathrm{d}s}Y_{t}^{F})^{1-\gamma}}{1-\gamma}\mathrm{d}t\nonumber\\
		&=\exp\Big(-\int_{0}^{t}\big(\delta^{\psi}\theta+\frac{\Phi(1-\gamma)}{2(\gamma+\Phi)^{2}}\|\eta\|^{2}e^{\delta\theta s}\big)\mathrm{d}s\Big) \widetilde{Z}_{t}\mathrm{d}B_{t}^{\mathbb{Q}^{\widehat{\xi}}}.
	\end{align*}
	Hence, using the fact that
	\begin{align*}
		&f\big(\widehat{c}_{t},e^{-\delta\theta t}\frac{(\widehat{X}_{t}+e^{rt}Y_{t})^{1-\gamma}}{1-\gamma}\big)+\frac{\Phi}{2(\gamma+\Phi)^{2}}\|\eta\|^{2}(\widehat{X}_{t}+e^{rt}Y_{t})^{1-\gamma}\\
		&=\Big(\delta^{\psi}\theta+\frac{\Phi(1-\gamma)}{2(\gamma+\Phi)^{2}}\|\eta\|^{2}e^{\delta\theta t}\Big)e^{-\delta\theta t} \frac{(\widehat{X}_{t}^{F}+e^{\int_{0}^{t}r_{s}\mathrm{d}s}Y_{t}^{F})^{1-\gamma}}{1-\gamma}
	\end{align*}
	for $t\in[0,T]$, and $Y_{T}=-e^{-rT}G$ we have (recall the definition of $h$ just below \eqref{Epstein-Zin generator})
	\begin{align*}
		\frac{(\widehat{X}_{0}+Y_{0})^{1-\gamma}}{1-\gamma}&=\mathbb{E}\Big[h(\widehat{X}_{T}-G)+\int_{0}^{T}\Big(f\big(\widehat{c}_{t},e^{-\delta\theta t}\frac{(\widehat{X}_{t}+e^{rt}Y_{t})^{1-\gamma}}{1-\gamma}\big)\\
		&\phantom{XXXXXXXXXXXx}+\frac{\Phi}{2(\gamma+\Phi)^{2}}\|\eta\|^{2}(\widehat{X}_{t}+e^{rt}Y_{t})^{1-\gamma}\Big)\mathrm{d}s\Big].
	\end{align*}
	Hence the upper bound in Lemma~\ref{Upper bound_Value function} is attained by $(\widehat{c},\widehat{\pi}^{S},\widehat{\pi}^{re},\widehat{\xi})$. We conclude that $(\widehat{c},\widehat{\pi}^{S},\widehat{\pi}^{re},\widehat{\xi})$ is optimal.
\end{proof}

	\section*{Acknowledgments}
	We would like to acknowledge fruitful discussions with Prof. Olivier Menoukeu Pamen.


\begin{thebibliography}{10}
	\bibitem{asmussen2020risk} {\sc Asmussen, S., Steffensen, M.}:
	{\it  Risk and Insurance.} Springer, Berlin, 2020.
	
	\bibitem{bauerle2005benchmark} {\sc B{\"a}uerle, N.}: {\it Benchmark and mean-variance problems for insurers.} Mathematical Methods of Operations Research, 62: 159--165, 2005.
	
	\bibitem{chen2020robust} {\sc Chen, Z., Yang, P.}: {\it Robust optimal reinsurance--investment strategy with price jumps and correlated claims.} Insurance: Mathematics and Economics, 92: 27--46, 2020.
	
	\bibitem{epstein1989substitution} {\sc Epstein, L.G., and Zin, S.E.}:
	{\it  Substitution, risk aversion, and the temporal behavior of consumption and asset returns: A theoretical framework.}  Econometrica, 57: 937–969, 1989.
	
	\bibitem{hansen2001robust} {\sc Hansen, L., Sargent, T.}:
	{\it  Robust control and model uncertainty.}  American Economic Review, 91: 60--66, 2001.
	
	\bibitem{hu2005utility} {\sc Hu, Y., Imkeller, P., M{\"u}ller, M.}:
	{\it  Utility maximization in incomplete markets.}  Annals of Applied Probability, 15: 1691--1712, 2005.
	
	\bibitem{kk2025optimal} {\sc Kuissi-Kamdem, W.}: {\it Asset-liability management with Epstein-Zin utility under stochastic interest rate and unknown market price of risk.} hal-05345383, 2025.
	
	\bibitem{ma2023optimal} {\sc Ma, J., Lu, Z., Chen, D.}: {\it Optimal reinsurance-investment with loss aversion under rough Heston model.} Quantitative Finance, 23: 95--109, 2023.
	
	\bibitem{maenhout2004robust} {\sc Maenhout, P.J.}: {\it Robust portfolio rules and asset pricing.} Review of Financial Studies, 17: 951--983, 2004.
	
	\bibitem{schmidli2007stochastic} {\sc Schmidli, H.}: {\it  Stochastic Control in Insurance.} Springer-Verlag, London, 2008.
	
	\bibitem{xie2020exploration} {\sc Xie, B., Yu, Z.}: {\it An exploration of Lp-theory for forward-backward stochastic differential equations with random coefficients on small durations.} Journal of Mathematical Analysis and Applications, 483: 123642, 2020.
	
	\bibitem{xing2017consumption} {\sc Xing, H.}: {\it Consumption--investment optimization with Epstein--Zin utility in incomplete markets.} Finance and Stochastics, 21: 227--262, 2017.
\end{thebibliography}
\end{document}